\documentclass[12pt]{article}  % Specifies the document style.

\usepackage{amsmath,amssymb,amsthm}
\usepackage[latin1]{inputenc}

\newtheorem{theorem}{Theorem}[section]
\newtheorem{lemma}[theorem]{Lemma}
\newtheorem{proposition}[theorem]{Proposition}

\newtheorem{remarks}[theorem]{Remarks}
\newtheorem{remark}[theorem]{Remark}
\theoremstyle{corollary}
\newtheorem{corollary}[theorem]{Corollary}

\def\N{\mathbb N}
\def\C{\mathbb C}
\def\R{\mathbb R}
\def\D{\mathbb D}
\def\T{\mathbb T}

\def\K{\mathbb K}
\def\ve{\varepsilon}

\def\al{\alpha}
\def\la{\lambda}
\def\dis{\displaystyle}
\def\ovl{\overline}

\date{}

\begin{document}

\title{Lineability criteria, with applications}

\author{Luis Bernal-Gonz\'alez and Manuel Ord\'o\~{n}ez Cabrera}

\maketitle

{\footnotesize  %\scriptsize
{\sl
\centerline{Departamento de An\'alisis Matem\'atico. Facultad de Matem\'aticas.}
\centerline{Apdo.~1160. Avda.~Reina Mercedes, 41080 Sevilla, Spain.}
\centerline{E-mails: {\tt lbernal@us.es, cabrera@us.es}}
}}

\begin{abstract}
\noindent Lineability is a property enjoyed by some subsets within a vector space $X$.
A subset $A$ of $X$ is called lineable whenever $A$ contains, except for zero, an infinite dimensional
vector subspace. If, additionally, $X$ is endowed with richer structures, then the more stringent notions
of dense-lineability, maxi\-mal dense-lineability and spaceability
arise naturally. In this paper, se\-ve\-ral lineability
criteria are provided and applied to specific topological vector spaces, mainly function spaces. Sometimes, such criteria
furnish unified proofs of a number of scattered results in the related literature.
Families of strict-order integrable functions, hypercyclic vectors, non-extendable holomorphic mappings, Riemann non-Lebesgue integrable functions,
sequences not satisfying the Lebesgue dominated convergence theorem, nowhere ana\-ly\-tic functions, bounded variation functions, entire functions with fast growth and Peano curves, among others, are ana\-ly\-zed from the point of view of lineability.

\vskip .15cm


\vskip .15cm

\end{abstract}

\section{Introduction}

\quad In the last two decades there has been a crescent interest in the search of nice
algebraic-topological structures within sets (mainly sets of functions or sequences)
that do not enjoy themselves such structures. This paper wants to contribute to shed light
on this recent trend, by providing a number of general criteria that guarantee the existence of the
mentioned structures, with emphasis in maximal dense-lineability and spaceability. Definitions are
given below. For a recent survey on lineability, see \cite{BPS}.

\vskip .15cm

To this respect, let us recall some recent terminology introduced
in \cite{AGS}, \cite{Bay1}, \cite{Ber2} and \cite{GuQ}. Assume that
$X$ is a vector space (over $\K :=$ the real line $\R$ or the complex plane $\C$)
and that $\alpha$ is a cardinal number. Then a subset $A$ of $X$ is called
\begin{enumerate}
\item[$\bullet$] {\it lineable} if $A \cup \{0\}$ contains an
infinite dimensional vector subspace,
\item[$\bullet$] {\it $\alpha$-lineable} if $A \cup \{0\}$ contains an
$\alpha$-dimensional vector subspace (hence lineable means $\aleph_0$-lineable,
where $\aleph_0 = {\rm card}\,(\N )$ and $\N$ stands for the set of positive integers),
\item[$\bullet$] {\it maximal lineable} if $A$ is ${\rm dim}\,(X)$-lineable.
\end{enumerate}
If, in addition, $X$ is a topological vector space, then we say that $A$ is
\begin{enumerate}
\item[$\bullet$] {\it dense-lineable} or {\it algebraically generic}
whenever $A \cup \{0\}$ contains a dense vector subspace of $X$,
\item[$\bullet$] {\it maximal dense-lineable}
whenever $A \cup \{0\}$ contains a dense vector subspace $M$ of $X$
with dim$\,(M) =$ dim$\,(X)$,
\item[$\bullet$] {\it spaceable} if $A \cup \{0\}$ contains some
infinite dimensional closed vector subspace.
\end{enumerate}

\noindent Other interesting properties --such as algebrability, introduced in
\cite{APS}, additivity, introduced in \cite{Nat1,Nat2} (see also \cite{GaMS}), and moduleability \cite{GPS}--
will not be considered here.
Note that if $X$ is an infinite dimensional separable Baire topological
vector space then $\mathfrak c$, the cardinality of the continuum, is the maximal dimension allowed to
any vector subspace of $X$. In particular, spaceability implies maximal lineability in this case.

\vskip .15cm

In the subsequent sections of this paper, a number of sufficient conditions for
maximal dense-lineability and spaceability will be stated, see Sections 2-3. The results that are obtained
turn to be improvements of known criteria. Finally, in Section 4, our results will be applied to obtain
lineability statements, mainly in the setting of function spaces. It is also shown how a number of known
assertions about lineability can be proved by using our theorems.

\section{Maximal dense-lineability}

\quad Many examples of nonlinear sets containing large vector spaces have been given in the
literature. Perhaps one of the most outstanding is the Herrero-Bourdon theorem (see \cite{Bou,Her})
asserting that the set $HC(T)$ of hypercyclic vectors of a (continuous, linear) operator
$T:X \to X$ on a complex Banach space $X$ is dense-lineable (moreover, the dense subspace
obtained is $T$-invariant; the result was extended by B\`es \cite{Bes} and Wengenroth \cite{Wen}
to any real or complex topological vector space). Recall that an operator $T:X \to X$ is said to
be {\it hypercyclic} whenever it admits a dense orbit, that is, whenever there is a vector $x_0 \in X$
(called hypercyclic for $T$) such that the set $\{T^nx_0: \, n \in \N\}$ is dense in $X$
(see \cite{BaM} and \cite{GrP} for excellent surveys on this subject). Another nice example
was established by Aron, Garc\'{\i}a and Maestre  \cite{AGM} in 2001. Namely, if $G \subset \C$ is a domain
(i.e.~$G$ is nonempty, open and connected) and $H(G)$ is the space of holomorphic functions in $G$
(endowed with the compact-open topology) then Mittag-Leffler discovered in 1884 that the subfamily $H_e(G)$ of functions
which are holomorphic exactly at $G$ --that is, which are non-extendable holomorphically
across the boundary $\partial G$ of $G$-- is nonempty. The authors of  \cite{AGM} showed that $H_e(G)$ is both
dense-lineable and spaceable in $H(G)$ (in fact,
the result is given in \cite{AGM} for domains of holomorphy in $\C^N$).
We will go back on these subjects later.

\vskip .15cm

By adopting a wider point of view, one might believe that large topological size
always entails algebraic genericity (for instance, $HC(T)$ is residual if $T$ is hypercyclic
on an $F$-space $X$ and, as proved by Kierst and Szpilrajn \cite{KiS} in 1933, $H_e(G)$ is residual in $H(G)$).
This is far from being true.
As an example, let $\N_0 := \N \cup \{0\}$ and
$\al = (a_k) \in \C^{\N_0}$ be a sequence with $\limsup_{k \to \infty} |a_k|^{1/k} < +\infty$,
and define the associated diagonal operator $\Delta_\al$ as
$\Delta_\al : \sum_{k=0}^\infty f_k z^k \in H(\C ) \mapsto \sum_{k=0}^\infty a_k f_k z^k \in H(\C )$.
If $\{\al_n = (a_{k,n})_{k \geq 0}: \, n \in \N\}$ is dense in $\C^{\N_0}$ then the set
$A := \{f \in H(\C ): \, (\Delta_{\al_n}f)_{n \ge 1}$ is dense in $H(\C )\}$ is residual
in $H(\C )$, but $A$ is not even $2$-lineable \cite{BCP}.

\vskip .15cm

In 2005, Bayart \cite{Bay1} gave several useful dense-lineability criteria, but focused on
divergence and universality of operators. With the aim to include more general situations,
Aron {\it et al.} \cite{AGPS} and the first author \cite{Ber,Ber2} proved res\-pec\-ti\-ve\-ly the following theorems.
According to \cite{AGPS}, if $A$ and $B$ are subsets of a vector space $X$, then $A$ is said
to be {\it stronger} than $B$ provided that $A + B \subset A$.

\begin{theorem}\label{dense-lineable}
Assume that $X$ is a metrizable separable topological vector space. If \,$A$ and $B$ are subsets
of \,$X$ such that $A$ is lineable, $B$ is dense lineable and $A$ is stronger than $B$,
then $A$ is dense-lineable.
\end{theorem}

\begin{theorem}\label{dense-lineable-bis}
Assume that $X$ is a metrizable separable topological vector space.
Suppose that \,$\Gamma$ is a family of vector subspaces
of $X$ such that $\bigcap_{S \in \Gamma} S$ is dense in $X$. We have:
\begin{enumerate}
\item[\rm (a)] If $\al$ is an infinite cardinal number such that $\bigcap_{S \in \Gamma} (E \setminus S)$
is $\al$-lineable then it contains, except for zero, a dense vector subspace of dimension $\al$.
\item[\rm (b)] In particular, if \,$\bigcap_{S \in \Gamma} (E \setminus S)$ is lineable then it is
dense-lineable. And if \,$\bigcap_{S \in \Gamma} (E \setminus S)$ is maximal lineable then it
is maximal dense-lineable.
\end{enumerate}
\end{theorem}

The idea which is in the core of both results above is to obtain the desired dense subspace
by adding small vectors coming from a known li\-ne\-a\-ble set to the vectors of a dense subset.
Theorems \ref{dense-lineable} and \ref{dense-lineable-bis} have been used in \cite{AGPS,Ber,Ber2,BeB}
to show the following assertions (each space $C^p[0,1], \, C^\infty [0,1], \, L^p, \, H(\D )$
is endowed with its natural topology):
\begin{enumerate}
\item[$\bullet$] The set $ND[0,1]$ of continuous nowhere differentiable functions on $[0,1]$
as well as the set $DNM[0,1]$ of differentiable nowhere monotone functions on $[0,1]$
are dense in $C[0,1]$ \cite{AGPS}.
\item[$\bullet$] Let $p \in \N_0$. Then the class of functions $f \in C^p[0,1]$ such that
$f^{(p)}$ is nowhere differentiable on $[0,1]$ is dense-lineable in $C^p[0,1]$ \cite{AGPS,Ber}.
\item[$\bullet$] The set of $C^\infty$-functions on $[0,1]$ which are nowhere analytic is
dense-lineable in $C^\infty [0,1]$ \cite{AGPS,Ber}.
\item[$\bullet$] Let $(\Omega, {\cal M}, \mu )$ be a measure space and $p \in [1,\infty )$ such that
the Lebesgue space $L^p := L^p(\mu , \Omega )$ is {\it separable}. Denote
$L^p_{r\hbox{-}strict}
:= L^p \setminus \bigcup_{q \in (p,\infty ]} L^q$ $(p \ge 1)$, $L^p_{l\hbox{-}strict}
:= L^p \setminus \bigcup_{q \in [1,p)} L^q$ $(p > 1)$, and $L^p_{strict}
:= L^p \setminus \bigcup_{q \in [1,\infty ] \setminus \{p\}} L^q$ $(p > 1)$. We have: \hfil\break
\phantom{a} $\blacktriangleright$ \,$L^p_{r\hbox{-}strict}$ is maximal dense-lineable if \hfil\break
    \phantom{aaaaaaaaaaaaa} $\inf \{\mu (S): \, S \in {\cal M}, \, \mu (S) > 0\} = 0$. \hfill $[\alpha ]$ \hfil\break
\phantom{a} $\blacktriangleright$ \,$L^p_{l\hbox{-}strict}$ is maximal dense-lineable if \hfil\break
    \phantom{aaaaaaaaaaaaa} $\sup \{\mu (S): \, S \in {\cal M}, \, \mu (S) < \infty\} = \infty$. \hfill $[\beta ]$ \hfil\break
\phantom{a} $\blacktriangleright$ \,$L^p_{strict}$ is maximal dense-lineable if both $[\al ]$ and $[\beta ]$ hold. \hfil\break
(In fact, conditions
$[\al ]$, $[\beta ]$, $[\al ] + [\beta ]$ are respectively ne\-ce\-s\-sa\-ry in the just mentioned assertions,
because they are respectively equivalent to the non-vacuousness of $L^p_{l\hbox{-}strict}, \, L^p_{l\hbox{-}strict}, \, L^p_{strict}$,
thanks to a result by Romero \cite{Rom} and Subramanian \cite{Sub}; see also
\cite[Section 14.8]{Nie} and \cite[proof of Theorem 3.4]{Ber2}.)
In particular, for the Lebesgue measure on $[0,1]$ we obtain for all $p > 1$ that $L^p[0,1] \setminus \bigcup_{q \in [1,p)}$
is maximal dense-lineable \cite{Ber2} (see also \cite{AGPS} and \cite{MPPS}).
\item[$\bullet$] Let $\D := \{z \in \C : \, |z| < 1\}$, the open unit disc. The set of functions
$f \in H(\D )$ that are strongly annular (i.e.~$\limsup_{r \to 1} \min \{|f(z)|: \, |z| = r\} = +\infty$)
is maximal dense-lineable \cite{BeB}.
\end{enumerate}

The approach of Theorems \ref{dense-lineable}--\ref{dense-lineable-bis} can be used
to discover (maximal) dense-lineability in more (already known or new) cases. In order to undertake
the task in a more systematic way, we are going to strengthen the above theorems. Note that in the following Theorem
\ref{maximal dense-lineable} and Corollary \ref{maximal dense-lineable corollary} {\it neither metrizability
nor separability} are needed as a general assumption. On the contrary, we need disjointness of the subsets
$A, \, B$ in order to estimate the dimension of the subspaces obtained.

\begin{theorem}\label{maximal dense-lineable}
Assume that $X$ is a topological vector space. Let $A \subset X$. Suppose that there exists a subset $B \subset X$
such that $A$ is stronger than $B$ and $B$ is dense-lineable. We have:
\begin{enumerate}
\item[\rm (a)] If $A$ is $\al$-lineable and $X$ has an open basis $\cal B$ for its to\-po\-lo\-gy
such that ${\rm card} ({\cal B}) \le \al$, then $A$ is dense-lineable. If, in addition, $A \cap B = \emptyset$,
then $A \cup \{0\}$ contains a dense vector space $D$ with ${\rm dim} (D) = \al$.
\item[\rm (b)] If $X$ is metrizable and separable and $\al$ is an infinite cardinal number such that $A$
is $\al$-lineable, and $A \cap B = \emptyset$, then $A \cup \{0\}$ contains a dense vector space $D$ with ${\rm dim} (D) = \al$.
\item[\rm (c)] If the origin possesses a fundamental system $\cal U$ of neighborhoods with
${\rm card} ({\cal U}) \le {\rm dim} (X)$, $A$ is maximal lineable and $A \cap B = \emptyset$,
then $A$ is maximal dense-lineable. In particular, the same conclusion follows if $X$ is metrizable,
$A$ is maximal lineable and $A \cap B = \emptyset$.
\end{enumerate}
\end{theorem}

\begin{proof}
Observe that (b) is derived from (a) because if $X$ is metrizable and separable then it is second countable,
hence it has a countable open basis $\cal B$ for its to\-po\-lo\-gy. Therefore
${\rm card} ({\cal B}) = {\rm card} (\N ) = \aleph_0 \le \alpha$ because $\al$ is infinite, and (a) applies.

\vskip .15cm

Let us show that (c) is also a consequence of (a). For this, assume that $A, \, B$ and $\cal U$ are as in the
hypothesis of (c). Let $C$ denote a dense countable subset of $\K$, and let $\{u_i\}_{i \in I}$ an algebraic
basis of $X$, so that ${\rm card} (I) = {\rm dim} (X)$. Denote by ${\cal P}_f (I)$ the family of
nonempty finite subsets
of $I$. Since ${\rm card} ({\cal U}) \le {\rm dim} (X)$, we must have that ${\rm dim} (X)$ is not finite,
hence ${\rm card} ({\cal P}_f(I)) = {\rm card} (I) = {\rm dim} (X) \ge {\rm card} (C)$. Moreover,
${\rm card} (I^F) = {\rm card} (I)$ for any nonempty finite set $F$, and ${\rm card} (C \times I) =
{\rm card} (I)$. Now, it is easy to see that the family
$$
{\cal B} := \left\{U + \sum_{i \in F} \al_i u_i: \, U \in {\cal U}, \, \al_i \in C \hbox{ for all } i \in F, \, F \in {\cal P}_f (I) \right\}
$$
is an open basis for the topology of $X$. We have that
$$
{\rm card} ({\cal B}) \le
{\rm card } \big({\cal U} \times \bigcup_{F \in {\cal P}_f(I)} (C \times I)^F \big)
= {\rm card } \big({\cal U} \times \bigcup_{F \in {\cal P}_f(I)} I^F \big)
$$
$$
\le {\rm card } ({\cal U} \times {\cal P}_f(I) \times I)
= {\rm card } ({\cal U} \times I \times I) = \max \{{\rm card} ({\cal U}),{\rm card} (I)\} = {\rm dim} (X).
$$
Since $A$ is ${\rm dim} (X)$-lineable, by applying (a) again we obtain the first part of (c).
As for the second part, simply observe that if $X$ is metrizable then $\cal U$ can be chosen countable,
so ${\rm card} ({\cal U}) \le {\rm dim} (X)$ if ${\rm dim} (X)$ is infinite.
If ${\rm dim} (X)$ is finite then the conclusion is evident because $A \cup \{0\} = X$; indeed, every
vector subspace $M$ of a finite dimensional vector space $X$ such that ${\rm dim} (M) = {\rm dim} (X)$
must equal $X$.

\vskip .15cm

Thus, our only task is to prove (a). Suppose that $A$ is $\al$-lineable and that ${\rm card} ({\cal B}) \le \al$
for some open basis $\cal B$ of $X$. We are also assuming that $A+B \subset A$ and $B$ is dense-lineable.
It follows that there exist vector spaces $A_1,B_1$ such that $A_1 \subset A \cup \{0\}$, $B_1 \subset B \cup \{0\}$,
$B_1$ is dense in $X$ and ${\rm dim} (A_1) = \al \ge {\rm card} ({\cal B})$. Hence there are sets $I,J$, vectors
$a_i$ $(i \in I)$ and open sets $U_j$ $(j \in J)$, such that ${\rm card} (I) = \al$, $\{a_i\}_{i \in I}$
is a linearly independent system contained in $A_1$, ${\cal B} = \{U_j\}_{j \in J}$ and there exists a
surjective mapping $\varphi : I \to J$. By density, we can assign to each $j \in J$ a vector $b_j \in U_j \cap B_1$.
Fix $j \in J$. As $U_j-b_j$ is a neighborhood of $0$ and multiplication by scalars is continuous on $X$,
for each $i \in \varphi^{-1}(\{j\})$ there is $\ve_i > 0$ satisfying $\ve_i a_i \in U_j - b_j$, or $\ve_ia_i+b_j \in U_j$.
Define
$$
D := {\rm span} \, \{\ve_i a_i + b_{\varphi (i)}: \, i \in I\}.
$$
Then $D$ is a vector subspace of $X$. Since $\varphi$ is surjective, we can pick for each $j \in J$ and
index $i(j) \in I$ with $\varphi (i(j)) = j$. As $\{U_j\}_{j \in J}$ is an open basis and $v_{i(j)}a_{i(j)}+b_j \in U_j$
$(j \in J)$, these vectors form a dense subset of $X$. But $D$ contains these vectors, so $D$ is also dense.
Furthermore, if $x \in D$ then there are $p \in \N$, $(\la_1, \dots ,\la_p) \in {\K}^p \setminus \{(0, \dots ,0)\}$
and $i_1, \dots ,i_p \in I$ with
$$
x = \la_1 \ve_{i_1} a_{i_1} + \cdots + \la_p \ve_{i_p} a_{i_p} + \la_1 b_{\varphi (i_1)} + \cdots + \la_p b_{\varphi (i_p)}.
$$
Let define $u := \la_1 \ve_{i_1} a_{i_1} + \cdots + \la_p \ve_{i_p} a_{i_p}$ and
$y := \la_1 b_{\varphi (i_1)} + \cdots + \la_p b_{\varphi (i_p)}$. Then $y \in B_1 \subset B \cup \{0\}$,
and $u \in A_1 \setminus \{0\}$ because of the linear independence of the $a_i$'s. Hence
$u \in A$ and
$$x = u + y \in A + (B \cup \{0\}) \subset A \cup A = A.$$
Consequently, $D \setminus \{0\} \subset A$ and $A$ is dense-lineable.

\vskip .15cm

Finally, we suppose further that $A \cap B = \emptyset$. We want to prove that ${\rm dim} (D) = \al$ or, that is the same,
the vectors $x_i := \ve_i a_i + b_{\varphi (i)}$ $(i \in I)$ are linearly independent. With this aim,
consider a $p \in \N$ and two $p$-tuples $(\la_1, \dots ,\la_p) \in {\K}^p$ and $(i_1, \dots ,i_p) \in I^p$
such that $\sum_{j=1}^p \la_j x_{i_j} = 0$. Assume, by way of contradiction, that
$(\la_1, \dots ,\la_p) \ne (0, \dots ,0)$. Then $u+y=0$, where $u$ and $y$ are as in the preceding paragraph.
Hence $y \in A$ (because $y = -u \in A_1 \setminus \{0\} \subset A$) and $y \in B$
(because $y = -u \ne 0$, so $y \in B_1 \setminus \{0\} \subset B$), which implies $A \cap B \ne \emptyset$.
This contradicts the assumption $A \cap B = \emptyset$, and we are done.
\end{proof}

\begin{corollary} \label{maximal dense-lineable corollary}
Let $X$ be a topological vector space. Suppose that $\Gamma$ is a family of vector subspaces of $X$ such that
$\bigcap_{S \in \Gamma} S$ is dense in $X$. We have:
\begin{enumerate}
\item[\rm (a)] If \,$\bigcap_{S \in \Gamma} (X \setminus S)$ is $\al$-lineable and $X$ has an open basis
$\cal B$ for its to\-po\-lo\-gy such that ${\rm card} ({\cal B}) \le \al$ then $\bigcap_{S \in \Gamma} (X \setminus S)$
is dense-lineable
and, moreover, it contains a dense vector space $D$ with ${\rm dim} (D) = \al$.
\item[\rm (b)] If \,$X$ is metrizable and separable and $\al$ is an infinite cardinal number such that
$\bigcap_{S \in \Gamma} (X \setminus S)$ is $\al$-lineable, then $\bigcap_{S \in \Gamma} (X \setminus S)$ contains,
except for zero, a dense vector space $D$ with ${\rm dim} (D) = \al$.
\item[\rm (c)] If the origin possesses a fundamental system $\cal U$ of neighborhoods with ${\rm card} ({\cal U})
\le {\rm dim} (X)$ then $\bigcap_{S \in \Gamma} (X \setminus S)$ is maximal dense-lineable. The same conclusion
holds if $X$ is metrizable and $\bigcap_{S \in \Gamma} (X \setminus S)$ is maximal lineable.
\end{enumerate}
\end{corollary}

\begin{proof}
In order to apply Theorem  \ref{maximal dense-lineable}, it is enough to check that
$A := \bigcap_{S \in \Gamma} (X \setminus S)$ is stronger than $B := \bigcap_{S \in \Gamma} S$,
that $B$ is dense-lineable and that $A \cap B = \emptyset$. The last property is obvious,
whereas the dense-lineability of $B$ is trivial in view of its denseness and the fact that
$B$ is itself a vector space. As for the property $A + B \subset A$, consider $x \in A$, $y \in B$
and $z := x+y$. If $z \notin A$ then there exists $S \in \Gamma$ with $z \in S$. Then
$$x = z + (-y) \subset S - B \subset S - S = S$$
as $S$ is a vector subspace. Thus $x \notin A$,
a contradiction, which concludes the proof.
\end{proof}

Note that the same technique shows that the first part of Theorem \ref{dense-lineable-bis}(b) is a special
instance of Theorem 2.1.

\vskip .15cm

We want to conclude this section by establishing a simple characterization of dense-lineability for the complement
of a subspace.

\begin{theorem} \label{simple dense-lineability}
Let $X$ be a metrizable separable topological vector space and $Y$ be a vector subspace of \,$X$. If \,$X \setminus Y$ is
lineable then \,$X \setminus Y$ is dense-lineable. Consequently, both properties of lineability and dense-lineability
for $X \setminus Y$ are equivalent provided that \,$X$ has infinite dimension.
\end{theorem}

\begin{proof}
It is evident that $X \setminus Y$ is lineable if and only $Y$ has infinite algebraic codimension.
The assumptions imply that $X$ has a countable open basis $\{G_n: \, n \ge 1\}$. Assume that $X \setminus Y$ is lineable.
In particular, $Y \subsetneq X$. Then $Y^0 = \emptyset$, hence $X \setminus Y$ is dense. Therefore there is
$x_1 \in G_1 \setminus Y$. Since codim$(Y) = \infty$, we have ${\rm span} (Y \cup \{x_1\}) \subsetneq X$.
Then $({\rm span} (Y \cup \{x_1\}))^0 = \emptyset$.
It follows that there exists $x_2 \in G_2 \setminus {\rm span} (Y \cup \{x_1\})$. With this procedure, we get
recursively a sequence of vectors $\{x_n\}_{n \ge 1}$ satisfying
$$x_n \in G_n \setminus {\rm span} (Y \cup \{x_1,...,x_{n-1}\}) \quad (n \ge 1).$$
In particular, the set $\{x_n: \, n \ge 1\}$ is dense. Now, if we define $M := {\rm span} \{x_n: \, n \ge 1\}$
then $M$ is a dense vector space and $M \setminus \{0\} \subset X \setminus Y$.
\end{proof}

Plainly, the scope of this result is shorter than that of the remaining criteria of this paper, so its use yields
weaker assertions on lineability in the diverse examples given in Section 4. Yet Theorem \ref{simple dense-lineability}
is easy to apply. For instance, the set $A := \{f \in C(\R ): \, f$ is unbounded$\}$ is dense-lineable because
$A = X \setminus Y$ with $X = C(\R )$, $Y = \{$bounded continuous functions $\R \to \R\}$, $Y$ is a vector subspace
and $A$ is lineable; indeed, $A$ contains the vector space of all non-zero polynomials $P$ with $P(0)=0$.

\begin{remark} \label{Remark-nonseparable}
{\rm A partial complement of Theorem \ref{simple dense-lineability} is possible in the non-separable case. Namely,
by assuming the Continuum Hypothesis, we have: \hfil\break
{\it \phantom{aaaa} Let \,$X$ be a non-separable F-space and \,$Y$ be a closed separable \hfil\break
\phantom{aaaa} vector subspace of \,$X$. Then \,$X \setminus Y$ is maximal lineable.} \hfil\break
Indeed, let $Z$ be a vector space that is an algebraic complement of $Y$, so that $Z \setminus \{0\} \subset X \setminus Y$.
Note that ${\rm dim} \,(Y) \le \mathfrak{c} \le {\rm dim} \,(X) = {\rm dim} \,(Y) + {\rm dim} \,(Z)$.
If ${\rm dim} \, (Z) \le \aleph_0$ then $Z$, and so $X$ ($= Y + Z$), would be separable (a contradiction).
Hence ${\rm dim} \, (Z) \ge \mathfrak{c}$, which implies ${\rm dim} \, (Z) ={\rm dim} \, (X)$, and we are done.}
\end{remark}

\section{Spaceability}

\quad Up to date, there not exist many explicit general criteria of existence of large closed subspaces
within a subset of a topological vector space. In fact, most spaceability proofs on specific settings
have been done directly and constructively.

One has to go back to Wilansky (\cite{Wil}, 1975) to find what maybe was the first general criterium. He proved that
{\it if \,$Y$ is a closed vector subspace of a Banach space $X$,
then $X \setminus Y$ is spaceable if and only if \,$Y$ has infinite codimension}
(compare to Theorem \ref{simple dense-lineability}).
An improved version of this result,
where $X$ is allowed to be a Fr\'echet space, is ascribed by Kitson and Timoney
\cite[Theorem 2.2]{KitT} to Kalton. The authors of \cite{KitT} exploit it to obtain the
following assertion (see \cite[Theorem 3.3]{KitT}).

\begin{theorem} \label{Kitson-Timoney}
Let $Z_n$ $(n \in \N )$ be Banach spaces and $X$ a Fr\'echet space. Let $T_n : Z_n \to X$ be continuous linear mappings and $Y$ the linear
span of \,$\bigcup_n T_n(Z_n)$. If \,$Y$ is not closed in $X$ then the complement $X \setminus Y$ is spaceable.
\end{theorem}

Among other applications, the last result is used in \cite{KitT} to show spa\-ce\-a\-bi\-li\-ty of the set of
non-absolutely convergent power series in the disk algebra $A(\D )$ and of the family of non-absolutely $p$-summing
operators between certain pairs of Banach spaces.

\vskip .15cm

Recently, the authors of \cite{BeO} in their Theorem 2.2 have
stated a sufficient condition for spaceability on function Banach spaces.
Then this theorem is applied to prove that conditions $[\al ], \, [\beta ], \, [\al ] + [\beta ]$ given in Section 2 are
respectively equivalent to the spaceability of $L^p_{r\hbox{-}strict}$ (if $p \ge 1$), $L^p_{l\hbox{-}strict}$ (if $p>1$) and
$L^p_{strict}$ (if $p > 1$) (hence equi\-va\-lent
to the respective non-vacuousness of these sets). It is also used to show that, if $CBV[0,1]$
denotes the Banach space (under the norm $\|f\| = |f(0)|+ {\rm Var}\,(f)$)
of continuous functions $[0,1] \to \R$ with bounded variation and $AC[0,1]$
represents the subset of absolutely continuous functions, then set $CBV[0,1] \setminus AC[0,1]$ is spaceable in $CBV[0,1]$.
Incidentally, Wilansky's theorem provides us with a shorter proof of this fact: $AC[0,1]$ is a closed subspace of
$CBV[0,1]$ (see e.g.~\cite{Ada}) and has infinite codimension, because $CBV[0,1] = AC[0,1] \oplus S[0,1]$,
where $S[0,1]$ stands for the subspace of continuous bounded variation {\it singular} (that is, with derivative $0$ almost everywhere) functions.

\vskip .15cm

The main ingredient in the proof of \cite[Theorem 2.2]{BeO} is Nikolskii's theorem of characterization of basic sequences.
But Nikolskii's theorem turns to be true in the setting of F-spaces (recall that an F-space is
a complete metrizable topological vector space).
Namely, if $X$ is an F-space
and $\| \cdot \|$ is an F-norm defining the topology of $X$, then a sequence $(x_n) \subset X \setminus \{0\}$
is basic if and only if there is a constant $\alpha \in (0,+\infty )$ such that, for every pair
$r,s \in \N$ with $s \ge r$ and every finite sequence of scalars $a_1, \dots , a_s$, one has
$$\big\|\sum_{n=1}^r a_n x_n\big\| \leq \alpha \big\|\sum_{n=1}^s a_n x_n\big\|$$
(see \cite[Theorem 5.1.8, p.~67]{KaG}).
Recall that an {\it F-norm} on a vector space $X$ is a functional $\| \cdot \|:X \to [0,+\infty )$ satisfying,
for all $x,y \in X$ and $\la \in \K$, the following properties: $\|x+y\| \le \|x\| + \|y\|$;
$\|\la x\| \le \|x\|$ if $|\la | \le 1$; $\|\la x\| \to 0$ if $\la \to 0$; $\|x\| = 0$ only if $x=0$.

\vskip .15cm

Then we can establish the following theorem, that is an improvement of Theorem 2.2 in \cite{BeO}.
By ${\cal P}(\Omega )$ we represent, as usual, the family of subsets of a set $\Omega$, while
$\sigma (f)$ will denote the support of a function $f:\Omega \to \K$, that is, the set
$$\sigma (f) = \{x \in \Omega : \, f(x) \ne 0\}.  \eqno (1)$$
\begin{theorem} \label{BerOrd}
Let \,$\Omega$ \,be a nonempty set and $Z$ be a topological vector space on \,$\K$.
Assume that \ $X$ \,is an F-space on $\K$ consisting of \,$Z$-valued functions on \,$\Omega$
\,and that \,$\| \cdot \|$ \,is an
F-norm defining the topology of \,$X$. Suppose, in addition, that \,$S$ is a nonempty subset of \,$X$
\,and that \,${\cal S}:{\cal P}(\Omega ) \to {\cal P}(\Omega )$ \,is a set function with \,${\cal S}(A) \supset A$ \,for
all \,$A \in {\cal P}(\Omega )$ \,satisfying the following properties:
\begin{enumerate}
\item[\rm (i)] If $(g_n) \subset X$  satisfies \,$g_n \to g$ in $X$ then there is a subsequence $(n_k) \subset \N$ such that, for every $x \in \Omega$, $g_{n_k}(x) \to g(x)$.
\item[\rm (ii)] There is a constant $C \in (0,+\infty )$ such that \,$\|f+g\| \ge C \|f\|$ \,for all $f,g \in X$ with \,$\sigma (f) \cap \sigma (g) = \emptyset$.
\item[\rm (iii)] $\al f \in S$ for all $\al \in \K$ and all $f \in S$.
\item[\rm (iv)] If $f,g \in X$ are such that $f+g \in S$ and \,${\cal S}(\sigma (f)) \cap \sigma (g) = \emptyset$, then $f \in S$.
\item[\rm (v)] There is a sequence of functions $\{f_n\}_{n \ge 1} \subset X \setminus S$ such that
                ${\cal S}(\sigma (f_m))  \cap \sigma (f_n) = \emptyset$ for all $m,n$ with $m \ne n$.
\end{enumerate}
Then the set \,$X \setminus S$ \,is spaceable in $X$.
\end{theorem}

\begin{proof}
Let us show that $(f_n)$ is a basic sequence. Indeed, by (iii) one derives that $0 \in S$, so from (v) we get $f_n \ne 0$ for all $n$; moreover, for every pair $r,s \in \N$ with $s \ge r$ and any scalars $a_1, \dots , a_s$ it follows from (ii) and (v) [and the fact ${\cal S}(\sigma (f_n)) \supset \sigma (f_n)$ for all $n$] that
$$
\|\sum_{n=1}^s a_n f_n\| = \|\sum_{n=1}^r a_n f_n + \sum_{n=r+1}^s a_n f_n\| \ge
C \|\sum_{n=1}^r a_n x_n\|,
$$
because the supports of \,$\sum_{n=1}^r a_n f_n$ \,and \,$\sum_{n=r+1}^s a_n f_n$ \,have
empty intersection, since $\sigma (\sum_{n \in F} a_n f_n)
\subset \bigcup_{n \in F} \sigma (f_n)$ for every finite set $F \subset \N$.
According to Nikolskii's theorem, $(f_n)$ is a basic sequence (with basic constant $\alpha = 1/C$).

\vskip .15cm

In particular, the functions $f_n$ $(n \ge 1)$ are linearly independent. Consider the set
$$M := \ovl{{\rm span}}\,\{f_n: \,n \in \N\}.$$
It is plain that $M$ is a closed infinite-dimensional vector subspace of $X$. It is enough to show that $M \setminus \{0\} \subset X \setminus S$. To this end, fix a function $F \in M \setminus \{0\}$. Then there is a uniquely determined sequence $(c_n) \subset \K$ such that
$$
F = \sum_{n=1}^\infty c_n f_n = \| \cdot \|\hbox{--}\lim_{n \to \infty} \sum_{k=1}^n c_k f_k.
$$
Let $N = \min \{n \in \N : \,c_n \ne 0\}$. Then $F = c_Nf_N + h$, with $h = \| \cdot \|\hbox{--}\lim_{n \to \infty} h_n$ and $h_n := \sum_{k=N+1}^n c_k f_k$ $(n \ge N+1)$. Note that $\sigma (f_N) = \sigma (c_Nf_N)$ as $c_N \ne 0$. If $x \in {\cal S} (\sigma (c_Nf_N)) = {\cal S} (\sigma (f_N))$ then, by (v), $x \notin \sigma (f_k)$ for all $k > N$. Hence $h_n(x) = 0$ for all $n > N$. But, from (i), there is a subsequence $(n_k) \subset \N$ with $h_{n_k} \longrightarrow h$ pointwise. Thus $h(x) = 0$ or, that is the same, $x \notin \sigma (h)$. Therefore ${\cal S} (\sigma (c_Nf_N)) \cap \sigma (h) = \emptyset$. By way of contradiction, assume that $F \in S$. Since $F = c_Nf_N + h$, we obtain from (iv) that $c_Nf_N \in S$.
By applying (iii) we get $f_N = c_N^{-1}(c_N f_N)\in S$, which contradicts (v). Consequently, $F \in X \setminus S$, as required.
\end{proof}

\begin{remark}\label{BerOrd-nota}
{\rm Observe that, apart from degrading $X$ to be an F-space, we have also replaced the field $\K$ in the original Theorem 2.2 of \cite{BeO} by any
topological vector space $Z$. Moreover, ${\cal S} (A)$ was simply $A$ in such theorem.}
\end{remark}

Applications of Theorems \ref{Kitson-Timoney}-\ref{BerOrd} will be given in the next section.

\section{Applications}

\quad In this section we make a number of applications of the diverse results established in the last two sections.
Our attention is mainly focused on function spaces.

\subsection{$L^p$ spaces.}
We begin by showing that, if sufficiently
many $\mu$-disjoint measurable sets for higher dimensions are allowed, separability is no longer needed
in the result about $L^p$ stated after Theorem \ref{dense-lineable-bis}. The following result due to
Botelho {\it et al.} \cite[Theorem 2.3]{BCFPS}, where ${\rm dim} (L^p)$ is computed, comes in our help.

\begin{theorem} \label{Lp-dimension}
Let $p \in (0,+\infty )$ and $(\Omega, {\cal M}, \mu )$ be a measure space. Consider its entropy
${\rm ent} (\Omega ):= {\rm card} ({\cal M}_f/{\cal R})$,
where ${\cal M}_f := \{S \in {\cal M}: \, \mu (S) < \infty\}$ and \,$\cal R$ is the equivalence relation in \,${\cal M}_f$ given by
$$
C \,\, {\cal R} \,\,D \hbox{ \ if and only if \ } \mu ((C \setminus D) \cup (D \setminus C)) = 0.
$$
We have:
\begin{enumerate}
\item[\rm (i)] If \,${\rm ent} (\Omega ) > {\mathfrak c}$ then ${\rm dim} (L^p) = {\rm ent} (\Omega )$.
\item[\rm (ii)] If \,$\aleph_0 \le {\rm ent} (\Omega ) \le {\mathfrak c}$ then ${\rm dim} (L^p) = {\mathfrak c}$.
\item[\rm (iii)] If \,${\rm ent} (\Omega ) \in \N$ \,then there is \,$k \in \N$ \,such that \,${\rm ent} (\Omega ) = 2^k$ \,and \hfil\break
                  ${\rm dim} (L^p) = k$.
\end{enumerate}
\end{theorem}

A family ${\cal S} \subset {\cal M}$ is called {\it $\mu$-disjoint} whenever $\mu (C) > 0$ for all $C \in {\cal S}$
and $\mu (C \cap D) = 0$ for all different $C,D \in {\cal S}$. It is easy to see that if $\cal S$ is a
$\mu$-disjoint family with ${\cal S} \subset {\cal M}_f$ and ${\rm card} ({\cal S}) > \aleph_0$
then $[\beta ]$ (see Section 2) holds.

\vskip .15cm

Observe that we allow $0 < p < \infty$ in Theorem \ref{Lp-maximal dense-lineability} below because the results we use
from \cite{Ber2}, \cite{Nie}, \cite{Rom} and \cite{Sub} are also valid in the non-normed case $0 < p < 1$. Then for
$0 < p < \infty$ we redefine \,$L^p_{r\hbox{-}strict}
:= L^p \setminus \bigcup_{q \in (p,\infty ]} L^q$, $L^p_{l\hbox{-}strict}
:= L^p \setminus \bigcup_{q \in (0,p)} L^q$ \,and \,$L^p_{strict}
:= L^p \setminus \bigcup_{q \in (0,\infty ] \setminus \{p\}} L^q$.
Note that these definitions are consistent with the earlier ones for $p \ge 1$,
because $L^r \cap L^t \subset L^s$ if $r<s<t$.

\begin{theorem} \label{Lp-maximal dense-lineability}
Let $p \in (0,\infty )$ and let $(X,\Omega ,\mu )$ be a measure space. If ${\rm ent} (\Omega ) > {\mathfrak c}$, we assume
that there is a $\mu$-disjoint family \,${\cal S} \subset {\cal M}$ \,with \hfil\break ${\rm card} ({\cal S}) = {\rm ent} (\Omega )$.
We have the following assertions, where the maximal dense-lineability is meant to be in $L^p$:
\begin{enumerate}
\item[\rm (a)] If \,$[\alpha ]$ holds and ${\rm ent} (\Omega ) \le {\mathfrak c}$ then $L^p_{r\hbox{-}strict}$ is maximal dense-lineable.
\item[\rm (b)] If \,${\rm ent} (\Omega ) > {\mathfrak c}$ and $[\alpha ]$ holds for every restricted space $(S,{\cal M}|_S,\mu |_S)$
$(S \in {\cal S})$, that is,
$$\inf \{\mu (C): \, C \in {\cal M}, \, C \subset S, \, \mu (C) > 0\} = 0, \eqno [\alpha_S]$$
then $L^p_{r\hbox{-}strict}$ is maximal dense-lineable.
\item[\rm (c)] If \,$[\beta ]$ holds and ${\rm ent} (\Omega ) \le {\mathfrak c}$ then $L^p_{l\hbox{-}strict}$ is maximal dense-lineable.
\item[\rm (d)] If \,${\rm ent} (\Omega ) > {\mathfrak c}$ and ${\cal S} \subset {\cal M}_f$ then $L^p_{l\hbox{-}strict}$ is maximal dense-lineable.
\item[\rm (e)] If \,${\rm ent} (\Omega ) \le {\mathfrak c}$ and $[\alpha ]$ and $[\beta ]$ hold
               then $L^p_{strict}$ is maximal dense-lineable.
\item[\rm (f)] If \,${\rm ent} (\Omega ) > {\mathfrak c}$, ${\cal S} \subset {\cal M}_f$ and $[\alpha_S ]$ holds for all $S \in {\cal S}$
               then $L^p_{strict}$ is maximal dense-lineable.
\end{enumerate}
\end{theorem}

\begin{proof}
Let $X := L^p$. Define $A$ as $L^p_{r\hbox{-}strict}$ in cases (a)-(b), as $L^p_{l\hbox{-}strict}$ in cases (c)-(d), and as
$L^p_{strict}$ in cases (e)-(f). Our task is to show that $A$ is maximal dense-lineable in $X$. To this end,
we consider the set $B$ of step functions, that is,
$$B = {\rm span} \, \{\chi_M: \, M \in {\cal M}, \, \mu (M) < \infty\},$$
where $\chi_M$ stands, as usual, for the characteristic function of $M$. It is well known that $B$ is dense in $L^p$.
Therefore $B$ is dense-lineable because it is a vector space itself. Since $B \subset L^q$ for all $q > 0$, we have
that $A \cap B = \emptyset$. Now, recall that $L^p$ is metrizable. According to Theorem \ref{maximal dense-lineable}(c),
it is enough to show that $A$ is maximal lineable. In other words, we have
to exhibit a vector space $M \subset A \cup \{0\}$ with ${\rm dim} (M) = {\rm dim} (L^p)$.

\vskip .15cm

\noindent (a) We assume that ${\rm ent} (\Omega ) \le {\mathfrak c}$ and that $[\alpha ]$ holds.
By the latter condition, there exists a sequence $(S_n)$ of pairwise disjoint measurable sets with
$0 < \mu (S_n) < 1/2^n$ $(n \in \N )$, see \cite[pp.~233--235]{Nie}. Then ${\rm ent} (\Omega ) \ge \aleph_0$,
so $\aleph_0 \le {\rm ent} (\Omega ) \le {\mathfrak c}$. It follows from Theorem \ref{Lp-dimension} that
${\rm dim} (L^p) = {\mathfrak c}$. In \cite[Proof of Theorem 3.4]{Ber2} it is proved that the functions
$f_a:\Omega \to [0,\infty )$ $(a>0)$ given by
$$f_a =\sum_{n=1}^\infty {\chi_{S_n} \over n^{1/p} (\log (n+1))^{a/p} \mu (S_n)^{1/p}} \eqno (2)$$
form a linearly independent family in $L^p$ and that
$M := {\rm span} \{f_a: \, a \in (1,\infty )\} \subset L^p_{r\hbox{-}strict} \cup \{0\}$.
Finally, ${\rm dim} (M) = {\rm card} ((1,\infty )) = {\mathfrak c} = {\rm dim} (L^p)$.

\vskip .15cm

\noindent (b) Here ${\rm ent} (\Omega ) > {\mathfrak c}$, so ${\rm dim} (L^p) = {\rm ent} (\Omega )$.
As before, condition $[\alpha_S]$ entails the existence, for each $S \in {\cal S}$, of a sequence
$\{C_{n,S}\}_{n \ge 1} \subset {\cal M}|_S$ such that $C_{n,S} \cap C_{m,S} = \emptyset$ if $n \ne m$
and $0 < \mu (C_{n,S}) < 1/n$ for all $n \in \N$. Note that, due to the $\mu$-disjointness of ${\cal S}$,
we can assume that $C_{n,S} \cap C_{m,\widetilde{S}} = \emptyset$ whenever $(n,S) \ne (m,\widetilde{S})$.
The last property guarantees the linear independence of the family $\{f_S\}_{S \in {\cal S}}$, where
$$f_S =\sum_{n=1}^\infty {\chi_{C_{n,S}} \over n^{1/p} (\log (n+1))^{2/p} \mu (C_{n,S})^{1/p}}.$$
Again, $M := {\rm span} \{f_S: \, S \in {\cal S}\} \subset L^p_{r\hbox{-}strict} \cup \{0\}$ and
${\rm dim} (M) = {\rm card} ({\cal S}) = {\rm ent} (\Omega ) = {\rm dim} (L^p)$.

\vskip .15cm

\noindent (c) This time $[\beta ]$ holds and ${\rm ent} (\Omega ) \le {\mathfrak c}$. But $[\beta ]$
implies the existence of a sequence $(S_n)$ of pairwise disjoint measurable sets with $1 < \mu (S_n) < \infty$
$(n \in \N )$. In particular, ${\rm ent} (\Omega ) \ge \aleph_0$, so Theorem \ref{Lp-dimension}
yields again that ${\rm dim} (L^p) = {\mathfrak c}$. If now we proceed exactly as in
\cite[Proof of Theorem 3.4]{Ber2} then one gets that the functions $f_a$ $(1 < a <\infty )$
defined as in (2) span again the desired vector space $M \subset L^p_{l\hbox{-}strict} \cup \{0\}$.

\vskip .15cm

\noindent (d) We are assuming here that ${\rm ent} (\Omega ) > {\mathfrak c}$ and ${\cal S} \subset {\cal M}_f$.
Let us mimic part of the clever proof of Theorem 3.4 in \cite{BCFPS}. Since ${\rm card} ({\cal S})
= {\rm ent} (\Omega ) > {\mathfrak c}$ and $\{\mu (S): \, S \in {\cal S}\} \subset (0,\infty )$
(with ${\rm card} ((0,\infty )) = {\mathfrak c}$), there must be $\gamma \in (0,\infty )$ and a
subfamily ${\cal S}_0 \subset {\cal S}$ such that ${\rm card} ({\cal S}_0) = {\rm ent} (\Omega )$
and $\mu (S) = \gamma$ for all $S \in {\cal S}_0$. Since ${\cal S}_0$ is uncountable, there is a collection
$\{{\cal S}_i\}_{i \in I}$ (with ${\rm card} (I) = {\rm card} ({\cal S}_0) = {\rm ent} (\Omega )$)
of pairwise disjoint countable families ${\cal S}_i = \{S_{i,n}\}_{n \ge 1}$
such that $\bigcup_{i \in I} {\cal S}_i = {\cal S}_0$.
Observe that $\mu (S_{i,n}) = \gamma$ for all $(i,n) \in I \times \N$. For each $i \in I$,
define \,$f_i := \sum_{n=1}^\infty {\chi_{S_{i,n}} \over \gamma n^{1/p} (\log (n+1))^{2/p}}$.
From the fact that the supports of the $f_i$'s are mutually $\mu$-disjoint one infers that these
functions are linearly independent. This together with the equality
$$\int_\Omega |f_i|^t \, d\mu =
\sum_{n=1}^\infty {1 \over n^{t/p} (\log (n+1))^{2t/p}} \,\,\,\, (< \infty \hbox{ if and only if } t \ge p)
$$
yields that \,$M := {\rm span} \{f_i: \, i \in I\}$ \,is a vector space satisfying
${\rm dim} (M) = {\rm card} (I) = {\rm dim} (L^p)$ and $M \setminus \{0\} \subset L^p_{l\hbox{-}strict}$.

\vskip .15cm

\noindent (e) This part is achieved by combining appropriately the approaches of (a) and (b), as
similarly suggested in \cite[Proof of Theorem 3.4]{Ber2}.

\vskip .15cm

\noindent (f) Finally, assume that ${\rm ent} (\Omega ) > {\mathfrak c}$, ${\cal S} \subset {\cal M}_f$
and $[\alpha_S]$ is satisfied for all $S \in {\cal S}$. By the proofs of (b) and (d), and since any
uncountable set can be partitioned into two sets with the same cardinality, one obtains that there
are $\gamma \in (0,\infty )$, a set $I$ with ${\rm card} (I) = {\rm ent} (\Omega )$ and families
${\cal S}_0 = \{S_{i,n}\}_{i \in \N , n \in \N}$, ${\cal S}_{00} = \{C_{i,n}\}_{i \in \N , n \in \N}$
such that $0 < \mu (C_{i,n}) < 1/2^n$, $\mu (S_{i,n}) = \gamma$ $(n \in \N , \, i \in I)$,
$\mu (C_{i,n} \cap C_{j,m}) = 0 = \mu (S_{i,n} \cap S_{j,m})$ if $(i,n) \ne (j,m)$ and
$\mu (C_{i,n} \cap S_{j,m}) = 0$ for all $(i,n),(j,m) \in I \times \N$. Define
$D_{i,2n-1} := C_{i,n}$, $D_{i,2n} = S_{i,n}$,
$\gamma_{i,2n-1} := \mu (C_{i,n})^{1/p}$, $\gamma_{i,2n} = \gamma$ $(i \in I, \, n \in \N )$. The functions
\,$f_i := \sum_{n=1}^\infty {\chi_{D_{i,n}} \over n^{1/p} (\log (n+1))^{2/p} \, \gamma_{i,n}}$
\,are easily seen to be linear independent and to satisfy that
\,$M := {\rm span} \{f_i: \, i \in I\}$ \,fulfills $M \setminus \{0\} \subset L^p_{strict}$ and
${\rm dim} (M) = {\rm ent} (\Omega ) = {\rm dim} (L^p)$.
\end{proof}

\begin{remarks}
{\rm 1. Note that Corollary \ref{maximal dense-lineable corollary} could also have been used
in the last proof: take $X = L^p$, $\Gamma = \{L^p \cap L^q\}_{q \in T}$, with $T = (p,\infty ]$,
$(0,p)$ or $(0,p) \cup (p,\infty ]$.

\noindent 2. Assume that ${\rm ent} (\Omega ) > c$. According to \cite[Lemma 3.1]{BCFPS}, a sufficient
condition for the existence of a $\mu$-disjoint family ${\cal S} \subset {\cal M}_f$ is the existence of a cardinal number
$\zeta$ such that ${\mathfrak c} \le \zeta < {\rm ent} (\Omega )$ and satisfying that, for every $A \in {\cal M}_f$ with
$\mu (A) > 0$, there are at most $\zeta$ subsets of $A$ with positive measure belonging to different classes of ${\cal M}_f/{\cal R}$.

\noindent 3. In \cite[Theorem 4.4]{BCFPS} a measure space $(\Omega ,{\cal M}, \mu )$ is constructed satisfying that,
for every $p,q$ with $1 \le q < p$, $L^p \setminus L^q$ is {\it not} maximal lineable. In particular, $L^p_{l\hbox{-}strict}$ is not
maximal lineable either.}
\end{remarks}

Concerning spaceability, a number of authors have recently devoted much effort to find large closed subspaces within
special subsets of $L^p$ (for general o specific measures $\mu$ such as the Lebesgue measure or the counting measure),
in particular within sets of functions which are $p$-integrable but not $q$-integrable for some $p,q \in (0,+\infty ]$,
see for instance \cite{BBG}, \cite{BeO}, \cite{BPS}, \cite{BCFP}, \cite{BCFPS}, \cite{BDFP}, \cite{BFPS},
\cite{GMS}, \cite{GPS}, \cite{GKP1} and \cite{GKP2}. Specially, in \cite{BCFPS} (\cite{GKP2}, resp.) sufficient conditions are given
for $L^p_{l\hbox{-}strict} \cup \{0\}$ to contain a closed vector space with maximal dimension
(a closed vector space isometric to $\ell_p$, resp.), among other interesting results; and in \cite{BCFP} and \cite{BDFP}
spaceability properties of subsets of the sequence spaces $c_0(X), \, \ell_p (X)$ $(0 < p <\infty )$, or similar ones (where $X$ is an
infinite dimensional Banach space), are shown.

\vskip .15cm

In Section 3 we mentioned that, by using the Banach version of Theorem \ref{BerOrd}, it was proved in \cite{BeO}
that the spaceability of $L^p_{r\hbox{-}strict}$ ($p \ge 1$), $L^p_{l\hbox{-}strict}$ ($p>1$) and $L^p_{strict}$ ($p>1$) is respectively equivalent
to $[\al ]$, $[\beta ]$ and $[\al ] + [\beta ]$. Let us show how the spaceability of these sets
can also be extracted from Theorem \ref{Kitson-Timoney}. The three cases being analogue, we will prove the assertion only for
$A := L^p_{r\hbox{-}strict}$ (with $p \ge 1$) under $[\al ]$. By using that the convergence of a sequence $(f_k)$ in $(L^r,\| \cdot \|_r)$
carries the a.e.-pointwise convergence of some subsequence, it is easy to see that
$(L^p \cap L^q,\|\cdot \|_p + \|\cdot \|_q)$ is a Banach space for each $q>p$. Moreover,
the inclusion $j_q:(L^p \cap L^q,\|\cdot \|_p + \|\cdot \|_q) \hookrightarrow (L^p,\| \cdot \|_p)$ is (linear and) continuous.
Now, it is well known that $L^r \cap L^s \subset L^t$ whenever $0 < r < t < s \le \infty$.
It follows that $A$ can be written as $A = L^p \setminus \bigcup_{q > p} (L^q \cap L^p) = L^p \setminus \bigcup_{n \ge 1} (L^{p+1/n} \cap L^p) = L^p \setminus {\rm span} \big( \bigcup_{n \ge 1} (L^{p+1/n} \cap L^p) \big)$.
Note that $Y := {\rm span} \big( \bigcup_{n \ge 1} (L^{p+1/n} \cap L^p) \big)$ is not closed in $L^p$ because it is dense in $L^p$
(since it contains all step functions) and $L^p \ne Y$ (due to $[\al ]$).
Finally, in Theorem \ref{Kitson-Timoney} just take $X = L^p$, $Z_n = L^p \cap L^{p+1/n}$ and $T_n = j_{p+1/n}$ $(n \ge 1)$.

\vskip .15cm

Notice that that $[\alpha ]$ is satisfied by $\Omega = [0,1]$ endowed with the Lebesgue measure. Hence the result proved in the previous paragraph
yields in particular the spaceability of $L^p_{r\hbox{-}strict} [0,1]$,
so covering the main statement in \cite{BFPS} for $p \ge 1$. But the case $p>0$ is not covered because, to the best of our
knowledge, Theorem \ref{Kitson-Timoney} has not been given a proof when the $Z_n$'s are just F-spaces. Nevertheless,
by using Theorem \ref{BerOrd} (with $Z = \K$, ${\cal S}(A) = A$ for all $A \subset \Omega$ and the F-norm in $L^p$ given by
$\|f\| = (\int_\Omega |f|^p \, d\mu )^{1/p}$ if $1 \le p < \infty$, $\|f\| = \int_\Omega |f|^p \, d\mu $ if $0 < p < 1$, and
$\|f\| = {\rm ess \,\, sup} \, |f|$ if $p = \infty$) and taking into account that, as already noticed, for every $p>0$ the non-vacuousness of
$L^p_{r\hbox{-}strict}, \, L^p_{l\hbox{-}strict}, \, L^p_{strict}$ is respectively equivalent to $[\al ]$, $[\beta ]$ and $[\al ] + [\beta ]$,
we can mimic the proof of Theorem 3.3 in \cite{BeO} so as to conclude the following result, which settles the question of spaceability of
the three mentioned sets, even in the non-locally convex case.

\begin{theorem}
Assume that $p \in (0,\infty ]$ and that $(\Omega ,{\mathcal M}, \mu )$ is a measure space. We have:
\begin{enumerate}
\item[\rm (a)] If \,$0 < p < \infty$, then \,$L^p_{r\hbox{-}strict}$ \,is spaceable if and only if \,$[\alpha ]$ holds,
$L^p_{l\hbox{-}strict}$ \,is spaceable if and only if \,$[\beta ]$ holds, and $L^p_{strict}$ is
spaceable if and only if both \,$[\alpha ]$ and \,$[\beta ]$ hold.
\item[\rm (b)] The set \,$L^\infty_{l\hbox{-}strict}$ \,is spaceable if and only if \,$[\beta ]$ holds.
\end{enumerate}
\end{theorem}

The Banach version of Theorem \ref{BerOrd} given in \cite{BeO} has been recently used by Akbarbaglu and Maghsoudi \cite{AkM}
to discover spaceability in certain related subsets of Orlicz spaces.

\vskip .15cm

In their paper \cite{GKP2} (see also \cite{GKP1}) Glab, Kaufmann and Pellegrini proved, among other results,
the following, which improves \cite[Theorem 4.1]{Ber2}.

\begin{theorem}\label{GlabKaufmannPellegrini}
Assume that $\mu$ is an atomless, outer regular, positive Borel measure on $\Omega$ with full support, where
$\Omega$ is a topological space admitting a countable family $(U_n)$ of nonvoid open subsets such that
every nonvoid open subset $A$ of \,$\Omega$ contains some $U_j$. Let $p \in (0,\infty )$ and consider the set
$S_p := \{f \in L^p: \, f$ is nowhere $L^q$ for each $q \in (p,\infty ]\}$. We have:
\begin{enumerate}
\item[\rm (a)] The set $S_p$ contains, except for zero, an $\ell_p$-isometric subspace of $L^p$. In particular, $S_p$ is spaceable.
\item[\rm (b)] The set $S_p$ is maximal dense-lineable.
\end{enumerate}
\end{theorem}

We recall that $f$ is nowhere $L^q$ means that, given a nonvoid open subset $U$ of $\Omega$,
the restriction $f|_U$ does not belong to $L^q(U)$. We notice that, once (a) is achieved, the proof of (b) given in \cite{GKP2}
can be considerably shortened by using Theorem \ref{maximal dense-lineable}(b): take $X = L^p$,
$\al = {\mathfrak c}$, $A = S_p$ and $B = \{$the step functions$\}$.

\begin{remark}
{\rm In view of the last argument (and others along this paper) one might believe that maximal dense-lineabiity
can only happen when there is spaceability. This is far from being true. For instance, for $X = c_0$ or $\ell_p$
$(1 \le p \le +\infty )$ Cariello and Seoane \cite{CaS} have recently proved that the subset
$$Z(X) := \{x=(x_n) \in X: \, x_n = 0 \, \hbox{ only for finitely many } \, n \in \N\}$$
is $\mathfrak{c}$-lineable
(so maximal lineable) but {\it not spaceable.} Now, if we take $A = Z(X)$ and $B = c_{00} =$ the space of
sequences with only finitely many nonzero entries, then $A + B \subset A$, and Theorem \ref{maximal dense-lineable}
yields that {\it $Z(X)$ is maximal dense-lineable.}}
\end{remark}

\subsection{Spaces of continuous and differentiable functions.}
Theorem \ref{maximal dense-lineable} can also be applied to reinforce other statements given after Theorem \ref{dense-lineable-bis}.
Recall that a $C^\infty$-function $f:[0,1] \to \K$ is said to have a {\it Pringsheim singularity} at a point $x_0 \in [0,1]$
whenever the radius of convergence of the Taylor series of $f$ at $x_0$ is zero. Obviously, in such a case, $f$ is not
analytic at $x_0$; but the converse is false.

\begin{theorem} \label{Cp-analytic-Pringsheim}
Consider the spaces \,$C^p [0,1]$ $(p \in \N_0 \cup \{\infty \})$ endowed with their natural topologies. We have:
\begin{enumerate}
\item[\rm (a)] Let $p \in \N_0$. The set $A_1 := \{f \in C^p[0,1]: \, f^{(p)}$ is differentiable
               at no point of $[0,1]\}$ is maximal dense-lineable in $C^p[0,1]$.
\item[\rm (b)] The set $A_2 := \{f \in C^\infty [0,1]: \, f$ is analytic at no point of $[0,1]\}$ is maximal dense-lineable in $C^\infty [0,1]$.
\item[\rm (c)] If \,$\K = \C$, the set $A_3 := \{f \in C^\infty [0,1]: \, f$ has a Pringsheim singularity at every point of \,$[0,1]\}$ is
               maximal dense-lineable in $C^\infty [0,1]$.
\end{enumerate}
\end{theorem}

\begin{proof}
Since the set $B := \{$polynomials$\}$ is dense in $C^p[0,1]$ for all $p \in \N_0 \cup \{\infty\}$, it is enough, according
to Theorem \ref{maximal dense-lineable}(b), to prove that $A_1, \, A_2$ and $A_3$ are ${\mathfrak c}$-lineable
(because $C^p[0,1]$ is metrizable, separable and complete, and ${\mathfrak c} = {\rm dim} (C^p [0,1])$ for all
$p \in \N_0 \cup \{\infty\}$). It is plain that $A_i \cap B = \emptyset$ for $i=1,2,3$.

\vskip .15cm

\noindent (a) From the spaceability of the set of nowhere differentiable functions in $C[0,1]$ (obtained by Fonf
{\it et al.} in \cite{FGK}) it follows the ${\mathfrak c}$-lineability of $A_1$ in the case $p=0$ (alternatively, see
a constructive proof in \cite{JMS}). Let $C$ be a ${\mathfrak c}$-dimensional vector subspace of \,$C[0,1]$ consisting, except for zero,
of nowhere differentiable functions on $[0,1]$. If $p \in \N$ and we let $\varphi_p$ denote the unique antiderivative of
order $p$ of a continuous function $\varphi :[0,1] \to \R$ such that $\varphi_p^{(k)}(0) = 0$ $(k \in \{0,...,p-1\})$, then it is easy to
see that $\{\varphi_p: \, \varphi \in C\}$ is a ${\mathfrak c}$-dimensional vector space contained in $A_1 \cup \{0\}$.

\vskip .15cm

\noindent (b) The maximal dense-lineability of $A_2$ is in fact established explicitly in \cite[Theorem 3.1]{Ber1},
where the ${\mathfrak c}$-lineability of $A_2$ is part of the proof: if $\varphi$ is nowhere analytic and
$e_{\alpha}(x) := e^{\al x}$ $(\al > 0)$ then ${\rm span} \{e_\al \, \varphi : \, \al > 0\} \subset A_2$ and
${\rm dim} ({\rm span} \{e_\al \, \varphi : \, \al > 0\}) = {\mathfrak c}$ (also Cater \cite{Cat} had obtained in 1984 such a
${\mathfrak c}$-dimensional space).

\vskip .15cm

\noindent (c) The ${\mathfrak c}$-lineability of $A_3$ is stated in \cite[Theorem 3.2]{Ber1} for $\K = \C$.
\end{proof}

Let us briefly turn our attention to {\it divergent Fourier series}. The existence of continuous functions
$f:\T \to \C$ ($\T := \{z=e^{it}: \, t \in [0,2\pi ]\}$, the unit circle) whose Fourier series
$\sum_{k=-\infty}^\infty \hat{f} (k) e^{ikt}$ diverges at some points is well known. Denote $S_n(f,t) :=
\sum_{k=-n}^n \hat{f} (k) e^{ikt}$ $(n \in \N )$, the partial Fourier sums. If $E \subset \T$, let
$${\cal F}_E :=
\{f \in C(\T ): \, \{S_n(f,t)\}_{n \ge 1} \hbox{ is unbounded for each } e^{it} \in E\}.
$$
In 2005, Bayart \cite{Bay1,Bay2} proved that, given $E \subset \T$ with Lebesgue measure zero, the set ${\cal F}_E$ is dense-lineable
and spaceable (Aron {\it et al.}~showed in \cite{APS} that ${\cal F}_E \cup \{0\}$ contains, in fact, an infinitely
generated dense algebra). If $E \subset \T$, consider the smaller set
$${\cal F}_{pE} := \{f \in C(\T ): \, \{(S_n(f,\cdot )|_E\}_{n \ge 1} \hbox{ is dense in } \C^E\},$$
where $\C^E$ is endowed with the topology of pointwise convergence. In 2010, M\"uller \cite{Mul}
proved that if $E$ is countable then ${\cal F}_{pE}$ is residual in $C(\T )$, while the first author \cite{Ber3}
demonstrated that ${\cal F}_{pE}$ is spaceable and maximal dense-lineable. We remark that, once the
spaceability is established, the maximal dense-lineability can be obtained from Theorem \ref{maximal dense-lineable}:
just choose $X = C(\T )$, $A = {\cal F}_{pE}$ and $B = \{$the trigonometric polynomials$\}$.

\vskip .15cm

\subsection{Holomorphic functions regular up to the boundary.}
A similar result holds for the family of non-continuable boundary-regular holomorphic functions. Assume that $G$ is
a domain in $\C$ and consider the Fr\'echet space $H(G)$ (a Fr\'echet space is a locally convex F-space).
Recall that for $f \in H(G)$ we have that $f \in H_e(G)$
if and only if, for all $z_0 \in G$, the radius of convergence $\rho (f,z_0)$ of the Taylor series of $f$ with center $z_0$
equals the euclidean distance $d(z_0,\partial G)$ between $z_0$ and the boundary $\partial G$ of $G$.
It was mentioned at the beginning of Section 2 that
Kierst and Szpilrajn showed the residuality of $H_e(G)$ (this result was extended by Kahane \cite{Kah} in 2000 to certain subspaces of $H(G)$)
and that in \cite{AGM} the dense-lineability and the spaceability of $H_e(G)$ were established (in \cite{Val} additional
topological pro\-per\-ties are found for the dense subspace within $H_e(G)$, and in \cite{BerS} spaces of holomorphic functions in
$\D$ are investigated). With more sophisticated methods --including the use of Arakelian's approximation theorem--
the first author (see \cite{BerB}) was able to state the ma\-xi\-mal dense-lineability of $H_e(G)$ in $H(G)$.

\vskip .15cm

Consider now the space $A^\infty (G)$ of {\it boundary-regular holomorphic functions} in $G$, that is, $f \in A^\infty (G)$
if and only if $f \in H(G)$ and each derivative $f^{(N)}$ $(N \ge 0)$ extends continuously on the closure $\overline{G}$ of $G$.
Then $A^\infty (G)$ can be endowed with a natural to\-po\-lo\-gy, namely, the topology of uniform convergence of functions
and all their derivatives on each compact subset of $\overline{G}$. In 1980 Chmielowski \cite{Chm} proved that if $G$
is regular (i.e.~$\overline{G}^0 = G$) then $A^\infty (G) \cap H_e(G)$ is nonempty, and finally Valdivia \cite{Val2} showed in 2009
that $A^\infty (G) \cap H_e(G)$ is in fact dense-lineable in $A^\infty (G)$.
By assuming additional conditions on $G$ (under which the authors of \cite{BCL} had obtained dense-lineability in 2008),
we are going to see that the last conclusion can be reinforced.
We say that a domain $G \subset \C$ is {\it finite-length} provided that there is $M \in (0,+\infty )$ such that
for any pair $a,b \in G$ there exists a curve $\gamma \subset G$ joining $a$ to $b$ for which ${\rm length} (\gamma ) \le M$.

\begin{theorem}
If $G \subset \C$ is a regular finite-length domain such that $\C \setminus \ovl{G}$ is connected then
$A^\infty (G) \cap H_e(G)$ is maximal dense-lineable in $A^\infty (G)$.
\end{theorem}

\begin{proof}
Firstly, let us prove that $A^\infty (G) \cap H_e(G)$ is maximal lineable. Since $G$ is regular, we can choose $\varphi \in
A^\infty (G) \cap H_e(G)$. Consider the functions $e_\al (z) := e^{\al z}$ $(\al > 0)$. The functions $e_\al \varphi$ are
linearly independent. Indeed, let $(c_1,...,c_N) \in \C^N \setminus \{(0,...,0)\}$ and
different $\al_1, \dots , \al_N \in (0,+\infty )$ such that $\sum_{i=1}^N c_i e_{\al_i} \varphi = 0$ on $G$.
Without loss of generality, we can assume that $N \ge 2$, $c_1 \ne 0$ and $\al_1 > \al_i$ if $i \ge 2$.
Since $\varphi \ne 0$ and $H(G)$ is an integrity domain, we have $\sum_{i=1}^N c_i e_{\al_i} = 0$ on $G$.
By the Identity Principle, we have $\sum_{i=1}^N c_i e^{\al_i z} = 0$ for all $z \in \C$.
In particular,
$$
c_1 + c_2e^{(\al_2-\al_1)x} + \cdots + c_Ne^{(\al_N-\al_1)x} = 0 \hbox{ \ for all } x > 0.
$$
Letting $x \to +\infty$, we get $c_1 + 0 = 0$, a contradiction, which shows the desired linear independence.
Now, set
$$M := {\rm span} \{e_\al \, \varphi : \, \al > 0\} \subset A^\infty (G).$$
Then ${\rm dim} (M) = {\mathfrak c} =
{\rm dim} (A^\infty (G))$. If $f \in M \setminus \{0\}$ then there are
$(c_1,...,c_N) \in \C^N \setminus \{(0,...,0)\}$ and
different $\al_1, \dots , \al_N \in (0,+\infty )$ such that $f = \sum_{i=1}^N c_i e_{\al_i} \varphi$.
Suppose, by way of contradiction, that $f \notin H_e(G)$. Let us denote by $S_{z_0}$ the sum of the Taylor
series of $f$ with center at $z_0$. Then there are a point $a  \in G$ and a number $r > d(a,\partial G)$
such that $S_a \in H(B(a,r))$. Of course, $S_a = f$ in $B(a,|a-b|)$, where $b$ is a point on $\partial G$ such
that $|a-b| = d(a,\partial G)$. Therefore there are a point $c \in \partial G$ and a number $\ve > 0$ with
$B(c,\ve ) \subset B(a,r)$ and $\sum_{i=1}^N c_i e^{\al_i z} \ne 0$ for all $z \in B(c,\ve )$;
indeed, $B(a,r)$ is a neighborhood of $b$, the point $b$ is not isolated in $\partial G$
(by the regularity of $G$), and the set of zeros of $\sum_{i=1}^N c_i e_{\al_i}$ in $\C$ is discrete.
Now take a point $\zeta \in B(c,\ve /2) \cap G$. Then $B(\zeta , \ve /2) \subset B(c,\ve ) \subset B(a,r)$ and
$\sum_{i=1}^N c_i e^{\al_i z} \ne 0$ for all $\zeta \in B(\zeta ,\ve /2)$. The function $S_\zeta$ equals $f$
in a neighborhood of $\zeta$, whence $S_\zeta /\sum_{i=1}^N c_i e_{\al_i}$ equals $\varphi$ in a neighborhood
of $\zeta$. We get from the non-extendability of $\varphi$ that
$$
{\ve \over 2} > d(\zeta ,c) \ge d(\zeta ,\partial G) = \rho (\varphi ,\zeta ) =
\rho \big({S_\zeta \over \sum_{i=1}^N c_i e_{\al_i}},\zeta \big) \ge {\ve \over 2}.
$$
This contradiction shows that $f \in H_e(G)$, so $M \setminus \{0\} \subset A^\infty (G) \cap H_e(G)$ and the
maximal lineability of the last set is guaranteed.

\vskip .15cm

According to \cite[Proof of Theorem 4]{MeN}, under the assumptions  on $G$
(specifically, $G$ is finite-length and $\C \setminus \ovl{G}$
is connected) the set of of polynomials is dense in $A^\infty (G)$.
Now, it is sufficient to apply Theorem \ref{maximal dense-lineable} with $X = A^\infty (G)$, $A = A^\infty (G) \cap H_e(G)$
and $B = \{$polynomials$\}$.
\end{proof}

\subsection{Sets of hypercyclic vectors.}
Our next application concerns hypercyclicity. The notion can be easily extended to sequences of operators, see
\cite{GrP}: given two (Hausdorff) to\-po\-lo\-gi\-cal vector spaces $X,Y$, a sequence $(T_n) \subset L(X,Y) :=
\{$continuous linear mappings $X \to Y\}$ is said to be {\it hypercyclic} provided that there is a vector
$x_0 \in X$ (called hypercyclic for $(T_n)$) such that the orbit $\{T_nx_0: \, n \in \N\}$ of $x_0$ under $(T_n)$
is dense in $Y$. We denote
$$HC((T_n)) = \{x \in X: \, x \hbox{ is hypercyclic for \ } (T_n)\}.$$
Note that if $X=Y$
and $T:X \to X$ is an operator (that is, $T \in L(X) := L(X,X)$), then $T$ is hypercyclic if and only if the
sequence $(T^n)$ of powers of $T$ is hypercyclic; moreover, $HC(T) = HC((T^n))$.
Only separable infinite dimensional topological vector spaces can support hypercyclic operators, see \cite{GrP}.
At the beginning of Section 2
we mentioned the Herrero-Bourdon-B\`es-Wengenroth theorem asserting the dense-lineability of $HC(T)$. The first
author \cite{BerU} proved that $HC(T)$ is {\it maximal} dense-lineable provided that $T$ is hypercyclic on a
{\it Banach} space (again, the dense subspace obtained in \cite{BerU} is $T$-invariant). As for sequences $(T_n) \subset L(X,Y)$,
it was demonstrated in \cite{BerA} that if $Y$ is metrizable and each subsequence $(T_{n_k})$
(with $n_1<n_2<\cdots$) is hypercyclic then $HC((T_n))$ is lineable, and that if $X$ and $Y$ are metrizable and
separable and $HC((T_{n_k}))$ is dense for each subsequence $(T_{n_k})$ of $(T_n)$ then $HC((T_n))$ is dense-lineable.
In Theorem \ref{hypercyclic-mdl} below it can be seen how spaceability, when it happens, comes in our help to obtain maximality.
But, prior to this, let us recall a recent, quantified version of hypercyclicity.

\vskip .15cm

According to Bayart and Grivaux \cite{BayG}, an operator $T$ on a topological vector space $X$ is said to be
{\it frequently hypercyclic} provided
there exists a vector \,$x_0 \in X$ \,such that
$$
\liminf_{n \to \infty} {{\rm card} \{k \in \{1,2,...,n\} : \, T^n x_0 \in U\} \over n} > 0
$$
for every nonempty open subset $U$ of $X$.
In this case, $x_0$
is called a frequently hypercyclic vector for $T$, and the set of these
vectors will be denoted by $FHC(T)$. The extension of the notion of frequent hypercyclicity to
sequences $(T_n) \subset L(X,Y)$ is obvious: replace $T^n$ by $T_n$ in the display above, and fix $U$ among the nonempty open subsets of $Y$. The corresponding set
of frequent hypercyclic vectors in $X$ is denoted by $FHC((T_n))$. In \cite{BayG} it is shown that if $T$ is a
frequent hypercyclic operator on a separable F-space $X$ then $FHC(T)$ is dense-lineable (once more, the
dense subspace obtained is $T$-invariant).

\begin{theorem} \label{hypercyclic-mdl}
\begin{enumerate}
\item[\rm (a)] Let $X$ be an infinite-dimensional separable F-space and \,$Y$ be a metrizable separable topological vector space.
Assume that $(T_n) \subset L(X,Y)$ and that
there is a dense subset $D \subset X$ such that the $(T_n)$ converges pointwise on $D$. If $HC((T_n))$ {\rm (}resp.~$FHC((T_n))${\rm )} is spaceable
then it is maximal dense-lineable.
\item[\rm (b)] Let $X$ be an F-space. Assume that $T \in L(X)$. If $HC(T)$ is spaceable and there is a sequence
$\{n_1 < n_2 < \cdots \} \subset \N$ such that $(T^{n_k})$ converges pointwise on some dense subset of $X$ then $HC(T)$ is maximal dense-lineable.
If \,$FHC(T)$ is spaceable and $(T^n)$ converges pointwise on some dense subset of $X$ then $FHC(T)$ is maximal dense-lineable.
\end{enumerate}
\end{theorem}

\begin{proof}
Part (b) follows from (a) by considering $T_n := T^n$ for frequent hypercyclicity, and taking $T_k := T^{n_k}$ for mere hypercyclicity,
together with the trivial fact $HC(T) \supset HC((T^{n_k}))$. Observe also that if $T$ is hypercyclic then $X$ must be separable
and infinite-dimensional.

\vskip .15cm

Let us prove (a). Suppose first that $HC((T_n))$ is spaceable. Then there is a closed infinite dimensional vector space $M \subset HC((T_n)) \cup \{0\}$. We have ${\rm dim} (X) = {\mathfrak c} = {\rm dim} (M)$ due to Baire's theorem. Therefore $A := HC((T_n))$ is maximal lineable. Note that
$B := \{x \in X: \, (T_nx)$ converges$\}$ is a vector space, and it is dense because $B \supset D$. Hence $B$ is dense-lineable. Now, trivially, if
a vector $x_0$ has dense orbit and $y_0 \in B$ then $\{T_n(x_0+y_0) = T_nx_0 + T_ny_0\}_{n \ge 1}$ is also dense, so $x_0+y_0 \in A$. In other words, $A + B \subset A$. Then $HC((T_n))$ is maximal dense-lineable by Theorem \ref{maximal dense-lineable}. Now, assume that $FHC((T_n))$ is spaceable. Take $A := FHC((T_n))$ and $B$ as before. Again by Theorem \ref{maximal dense-lineable}, the only property to show is $A+B \subset A$. Fix $x_0 \in A$ and $y_0 \in B$. Given a nonempty open set $U \subset Y$, choose any $u_0 \in U$ and a neighborhood $V$ of $0$ in $Y$ such that $V+V \subset U-u_0$.
Then $\liminf_{n \to \infty} {\rm card} \{k \in \{1,2,...,n\} : \, T_n x_0 \in V+u_0-z_0\}/n > 0$,
where $z_0 := \lim_{n \to \infty} T_ny_0$. Since $T_n y_0 \in V+z_0$ for
$n$ large enough, we obtain that $T_n (x_0 + y_0) = T_nx_0 + T_ny_0 \in (V+u_0-z_0)+(V+z_0) \subset U$ whenever $T_n x_0 \in V+u_0-z_0$ and $n$ is large enough. This yields $\liminf_{n \to \infty} {\rm card} \{k \in \{1,2,...,n\} : \, T_n (x_0+y_0) \in U\}/n > 0$, that is, $x_0 + y_0 \in A$.
\end{proof}

Now, we can establish a general existence result for {\it Fr\'echet} \,spaces.

\begin{corollary}
Let $X$ be a separable infinite dimensional Fr\'echet space. Then $X$ supports an operator $T$ such that $HC(T)$ is maximal dense-lineable.
\end{corollary}

\begin{proof}
In 1998, Bonet and Peris \cite{BoP} proved that if $X$ is as in the hypothesis then there exists $T \in L(X)$ such that
$T$ is hypercyclic. Recently, Menet \cite[Theorem 2.4]{Men2} has shown that $T$ can be chosen such that $HC(T)$ is spaceable.
If $X$ is not isomorphic to $\omega := \K^\N$, it is observed in \cite[Proof of Theorem 2.4]{Men2} that the operator $T$ obtained there
(which is based on the construction in \cite{BoP})
satisfies that $(T^n)$ converges pointwise on a dense set, so Theorem \ref{hypercyclic-mdl} applies. If $X$ is isomorphic to $\omega$, let $S:X \to \omega$ be such an isomorphism.
B\`es and Conejero \cite{BesC} demonstrated that for the backward shift $B:(x_1,x_2,...) \in \omega \mapsto (x_2,x_3,...) \in \omega$
one has that $HC(B)$ is spaceable. Trivially, $B^n \to 0$ poinwise on the dense subset $D_0 := \{(x_n) \in \omega : \, x_n \ne 0$ only for finitely many $n\}$. It follows that the operator $T := S^{-1}BS:X \to X$ satisfies that $HC(T)$ is spaceable and $T^n \to 0$ pointwise on the dense set $S^{-1}(D)$. A new application of Theorem \ref{hypercyclic-mdl} yields the conclusion.
\end{proof}

Of course, in order to apply Theorem \ref{hypercyclic-mdl}, it is important to have to our disposal a number of results on spaceability of
the set of hypercyclic/frequently hypercyclic vectors: the interest reader is referred to \cite{BaM}, \cite{BoG} and \cite{GrP}. As a first example,
note that the maximal dense-lineability of the family ${\cal F}_{pE}$ given in Subsection 4.2 may be obtained by using the mentioned theorem.

\vskip .15cm

Let us give examples of operators on non-Banach spaces whose sets of hypercyclic vectors are maximal dense-lineable.
To start with, we consider the space $H(\C^N)$ of entire functions $\C^N \to \C$, endowed with the compact-open topology.
Recall that each $a \in \C^N$ generates a
translation operator $\tau_a : f \in H(\C^N) \mapsto f( \cdot + a) \in H(\C^N)$. Also, if $D$ denotes the derivative operator
on $H(\C )$
(i.e.~$Df = f'$), then every polynomial $P(z) = a_0+a_1z+ \cdots +a_nz^n$ generates a finite order differential operator
$P(D) := a_0 I + a_1 D + \cdots + a_n D^n$, where $I$ is the identity operator.

\begin{proposition}
Assume that $T:H(\C^N) \to H(\C^N)$ is an operator that commutes with translations,
that is, $T \tau_a = \tau_a T$ for all $a \in \C^N$. Assume also that $T$ is not a scalar multiple of the identity. We have:
\begin{enumerate}
\item[\rm (a)] The set \,$HC(T)$ is maximal dense-lineable.
\item[\rm (b)] If \,$N > 1$ then \,$FHC(T)$ is maximal dense-lineable.
\item[\rm (c)] If \,$N=1$ and \,$T$ is not a finite order differential operator then \,$FHC(T)$ is maximal dense-lineable.
\end{enumerate}
\end{proposition}

\begin{proof}
In Corollary 2 of \cite{BoG} it is shown that if $N > 1$ (if $N=1$, resp.) then any non-scalar convolution operator
(any non-scalar convolution operator that is not $P(D)$ for any polynomial $P$, resp.) $T \in L(H(\C^N))$ satisfies
that $FHC(T)$ is spaceable. Since an operator is of convolution if and only if it commutes with translations
(see e.g.~\cite{GoS}) and since $HC(T) \supset FHC(T)$, we get spaceability for all sets in (a), (b), (c). Indeed,
the only case to consider in order to complete this claim is the spaceability of
$HC(P(D))$ whenever $N=1$ and $P$ is a nonconstant polynomial.
But this has been recently proved by Menet \cite{Men1}. According to
Theorem \ref{hypercyclic-mdl}, to conclude the proof it is enough to show that, for any operator $T$ as
in the statement of the theorem, $(T^n)$ converges pointwise on some dense subset of $X := H(\C^N)$.
Bonilla and Grosse-Erdmann (see \cite{BoG0} and \cite{BoG}) have proved that such a $T$
satisfies the so-called Frequent Hypercyclicity Criterion,
one of whose items is the existence of a dense subset $D \subset X$ such that $\sum_{n \ge 1} T^n x$
converges unconditionally for every $x \in D$. Clearly, this implies $T^n x \to 0$ for all $x \in D$, and we are done.
\end{proof}

A second example is provided by composition operators. Suppose that $G \subset \C$ is a domain and that
$\varphi :G \to G$ is a holomorphic self-mapping. Then $\varphi$ generates the composition operator
$C_\varphi : f \in H(G) \mapsto f \circ \varphi \in H(G)$. B\`es \cite[Theorem 1]{Bes2} has proved that if $\varphi$
is one-to-one and has no fixed point in $G$ then, for every nonconstant polynomial $P$,
the set \,$FHC(P(C_\varphi ))$ is spaceable. Since in \cite{Bes2} it is shown that every such $P(C_\varphi )$
satisfies the Frequent Hypercyclic Criterion, we get $P(C_\varphi )^n \to 0$ on a dense set and, by Theorem \ref{hypercyclic-mdl},
the set \,$FHC(P(C_\varphi ))$ is maximal dense-lineable. Finally, similar arguments allow us to assert
(under appropriate conditions) maximal
dense-lineability for $HC(B_w)$, where $B_w:(x_n) \in X \mapsto (w_nx_{n+1}) \in X$ is the backward shift
with weight sequence $(w_n)$ acting on a K\"othe sequence space $X = \lambda^p(A)$ or $c_0(A)$ $(1 \le p < \infty )$,
where $A = (a_{j,k})_{j,k \ge 1}$ is a matrix such that $a_{j,k} > 0$ and $a_{j,k} \le a_{j+1,k}$ for any $j,k \ge 1$
(see \cite{Men1} for conditions guaranteeing spaceability of $HC(B_w)$ in these spaces).

\subsection{Functions of bounded variation.}

In Section 3 we considered the space $CBV[0,1]$ of functions $f:[0,1] \to \R$ which are continuous and of bounded variation,
endowed with the norm $\|f\| := |f(0)|+{\rm Var}\,(f)$. Recall the decomposition $CBV[0,1] = AC[0,1] \oplus S[0,1]$.
Observe that the latter norm is strictly finer that the maximum norm. In fact, the Banach space $(CBV[0,1],\| \cdot \|)$
is nonseparable (see e.g.~\cite{Ada}; see also a nice proof in \cite[Section 1]{BBF}). Recall that a function $f \in CBV[0,1]$ is said
to be {\it strongly singular} whenever $f \in S[0,1]$ (that is, $f'=0$ a.e.~on $[0,1]$) and $f$ is nonconstant on any subinterval of $[0,1]$.
In particular, every strongly singular function is not absolutely continuous in any subinterval of $[0,1]$. The set of these
functions will be denoted by $SS[0,1]$.

\vskip .15cm

Recently, Jim\'enez-Rodr\'{\i}guez \cite{Jim} has shown that $c_0$ is isometrically isomorphic to a
subspace of continuous functions $[0,1] \to \R$ all of whose nonzero members are non-Lipschitz and
have a.e.~null derivative, so improving a result due to Jim\'enez {\it et al.}~\cite{JMS-a} asserting
the $\mathfrak c$-lineability of this family of functions. In particular, this family is spaceable in $C[0,1]$.
Notice that eve\-ry member of $SS[0,1]$ belongs to the described family. In \cite{BBF}, Balcerzak
{\it el al.} have demonstrated the spaceability of $SS[0,1]$ in $CBV[0,1]$ (in fact, a nonseparable closed
subspace is found in $SS[0,1] \cup \{0\}$). This improves the result of spaceability of $CBV[0,1] \setminus AC[0,1]$,
see Section 3. In \cite[Theorem 10]{BBF} it is proved an assertion containing the maximal dense-lineability of
$SS[0,1]$ in $C[0,1]$. In particular, the family $\cal A$ of functions $f \in CBV[0,1]$ being not absolutely continuous
in any subinterval of $[0,1]$ is maximal dense-lineable in $C[0,1]$. This result can also be deduced from the
mentioned
spaceability of $SS[0,1]$ (which implies the spaceability of $\cal A$, and so the $\mathfrak c$-lineability of $\cal A$)
together with the density of the set $B$ of polynomials in $C[0,1]$: just observe that ${\cal A} + B \subset {\cal A}$ and
apply Theorem \ref{maximal dense-lineable}.

\begin{remark}
{\rm In \cite[Proposition 3.3]{Ber} it is asserted the dense-lineability {\it in} $CBV[0,1]$ of
the family of functions $f \in CBV[0,1]$
which are differentiable on no interval in $[0,1]$, while in \cite[Theorem 4.2]{BeO} it is established the (stronger)
assertion of maximal dense-lineability of $\cal A$ {\it in} $CBV[0,1]$. Unfortunately, both proofs were based on the density of the set of
polynomials in $CBV[0,1]$, which is {\it false.} Consequently, the mentioned assertions are not proved (nor disproved, as
far as we know) up to date. We apologize for this.}
\end{remark}

\subsection{Riemann-integrable functions on unbounded intervals.}

Let $I \subset \R$ be an unbounded interval. Consider the Lebesgue space $L^1(I)$,
the Banach space $B(I)$ of all bounded functions $I \to \R$ (endowed with the supremum norm), and the
vector space $R(I)$ of all Riemann-integrable functions on $I$. On the one hand,
Garc\'ia-Pacheco, Mart\'in and Seoane \cite{GMS} es\-ta\-bli\-shed in 2009 the spaceability in $B(I)$ of the set of
all continuous bounded functions on $I$ which are not Riemann-integrable, as well as
the spaceability in \,$L^1(I)$ of \,$L^1(I) \setminus R(I)$ (see also \cite{GGMS}, \cite{GKP1} and
\cite[Section 2.4]{BPS}; the last result has been recently improved in \cite{GKP2}, as
mentioned in Theorem \ref{GlabKaufmannPellegrini} above: take $\Omega = I$, $p=1$). On the other hand, the existence of
Riemann-integrable functions on a given unbounded interval being not Lebesgue-integrable on it is well known:
consider the classical example $f(x) = \dis{\sin x \over x}$ on $(0,+\infty )$. In fact, in \cite{GMS} it is proved
the lineability of $R(I) \setminus L^1(I)$. In this context, an arising natural question is whether
this lineability can be enriched within some appropriate topological structure.
Theorem \ref{Riemann} below provides a po\-si\-ti\-ve answer, but we need a preliminary lemma. If $I \subset \R$
is an unbounded interval, we denote by \,$C_0(I)$ \,the space of all continuous
functions \,$f:I \to \R$ \,such that \,$\lim_{x \to \infty} f(x) = 0$. For the sake of
simplicity, we will only consider the case $I = [0,+\infty)$, the remaining ones being analogue.

\begin{lemma} \label{Riemann-lemma}
Let \,$I = [0,+\infty )$. Then the expression
$$\|f\| := \sup_{x \ge 0} |f(x)|  + \sup_{x \ge 0} \left| \int_0^x f(t) \, dt \right|$$
defines a norm on the space \,$C_0(I) \cap R(I)$ \,which makes it a separable Banach space.
The set \,$B$ of continuous functions \,$[0,+\infty ) \to \R$ with bounded support is a dense
vector subspace of this space.
\end{lemma}

\begin{proof}
The linearity of the integral together with the fact that $\sup_{x \ge 0} |f(x)|$ is a norm on $C_0(I)$
yields that $\| \cdot \|$ is a norm on $X := C_0(I) \cap R(I)$. Let us prove that $(X, \| \cdot \|)$ is
complete. If $(f_n)$ is a $\| \cdot \|$-Cauchy sequence in $X$ then it is, trivially, a Cauchy sequence
in the Banach space $C_0(I)$ endowed with the supremum norm. Hence there is $f \in C_0(I)$ such that
$f_n \to f$ uniformly on $I$. We need to show that $f \in R(I)$ and $f_n \to f$ for $\| \cdot \|$.
To this end, fix $\ve > 0$. There is $N \in \N$ such that $\|f_m-f_n\| < \ve /3$ for all $m \ge n \ge N$.
Then \,$|\int_0^x (f_m(t)-f_n(t)) \,dt| < \ve /3$ \,for all $x > 0$ and all $m \ge n \ge N$. In particular,
setting $n = N$ and letting $m \to \infty$ one gets by invoking uniform convergence on $[0,x]$ that
\,$|\int_0^x (f(t)-f_N(t)) \,dt| \le \ve /3$ \,for all $x > 0$. Since $f_N \in R(I)$ there is $a > 0$ such that
\,$|\int_b^c f_N(t) \,dt| < \ve /3$ \,for all $c > b > a$. It follows from the triangle inequality that
\begin{equation*}
\begin{split}
\left|\int_b^c f(t) \,dt\right| \le &\left|\int_b^c f_N(t) \,dt\right| + \left|\int_0^b (f(t)-f_N(t)) \,dt\right| + \\
                         &\left|\int_0^c (f(t)-f_N(t)) \,dt\right| < {\ve \over 3} + {\ve \over 3} + {\ve \over 3} = \ve ,
\end{split}
\end{equation*}
and Cauchy's criterion for improper Riemann integrals guarantees that $f \in R(I)$.
Now, we have $|\int_0^x (f(t)-f_n(t)) \,dt | \le \ve /3$ for all $n \ge N$ and all $x > 0$, and $N$ can be chosen
so that $\sup_{x>0} |f(x) - f_n(x)| < \ve /2$ for all $n \ge N$. Therefore $\|f_n -f\| < \ve$ if $n \ge N$,
which shows that $f_n \mathop{\longrightarrow}\limits_{n}^{\| \cdot \|} f$.
Hence \,$(X,\| \cdot \|)$ \,is complete.

\vskip .15cm

Next, consider the set $B$ defined in the statement of this lemma.
It is clear that $B$ is a vector subspace of $X$. Fix a function $f \in X$
as well as an $\ve > 0$. Then there is $a > 0$ with $|f(x)| < \ve /6$ $(x > a)$ and
$|\int_b^c f(t) \, dt| < \ve /6$ for all $c > b > a$. Define $f_a:[0,+\infty ) \to \R$ as $f_a = f$ on $[0,a]$,
$f_a = 0$ on $(a+1,+\infty )$, and $f_a$ affine-linear on $[a,a+1]$ with $f_a(a+1)=0$.
It follows that $f_a \in B$ and $\|f-f_a\| \le \sup_{a < x <a+1} |f(x)-f_a(x)| + \sup_{x \ge a+1} |f(x)| + |\int_a^{a+1} f(t) \,dt| +
|\int_a^{a+1} f_a(t) \,dt | + \sup_{x>a+1} |\int_{a+1}^x f(t) \,dt| \le 5 \cdot {\ve \over 6} < \ve$.
Thus $B$ is dense in $X$.
Finally, by using the Weierstrass polynomial approximation theorem it is not difficult to realize that
the countable set $\{f_a: \, a \in \N$ and $f$ is a polynomial with rational coefficients$\}$
is dense in $X$, so yielding the separability of $X$.
\end{proof}

\begin{theorem} \label{Riemann}
Let \,$I = [0,+\infty )$. Then the set
$$C_0(I) \cap R(I) \setminus \bigcup_{0<p<\infty} L^p(I)$$
is spaceable and maximal dense-lineable in \,$(C_0(I) \cap R(I),\| \cdot \|)$.
\end{theorem}

\begin{proof}
Set $X = (C_0(I) \cap R(I),\| \cdot \|)$ and $A = C_0(I) \cap R(I) \setminus \bigcup_{0<p<\infty} L^p(I)$,
and consider the set $B$ in Lemma \ref{Riemann-lemma}. Then $X$ is metrizable, separable and, plainly,
$A \cap B = \emptyset$ and $A$ is stronger than $B$. By Lemma \ref{Riemann-lemma}, $B$ is dense-lineable.
If we proved the spaceability of $A$ then we would obtain that $A$ is maximal lineable (because $X$ is separable),
so it would follow from Theorem \ref{maximal dense-lineable} that $A$ is maximal dense-lineable. Therefore
it suffices to demonstrate that $A$ is spaceable.

\vskip .15cm

To this end, we will try to apply Theorem \ref{Kitson-Timoney}. Set $Y = C_0(I) \cap R(I) \cap \bigcup_{0<p<\infty} L^p(I)$,
so that $A = X \setminus Y$. Since $C_0(I) \cap L^p(I) \subset C_0(I) \cap L^q(I)$
whenever $q \ge p$, we get
\begin{equation*}
\begin{split}
Y &= C_0(I) \cap R(I) \cap \bigcup_{n=1}^\infty L^n(I) \\
  &= {\rm span} \big(C_0(I) \cap R(I) \cap \bigcup_{n=1}^\infty L^n(I) \big)
= {\rm span} \big( \bigcup_{n=1}^\infty T_n(Z_n) \big) ,
\end{split}
\end{equation*}
where $Z_n = C_0(I) \cap R(I) \cap L^n(I)$ and $T_n$ denotes the inclusion $Z_n \hookrightarrow X$.
It is plain that $T_n$ is (linear and) continuous if each $Z_n$ is endowed with the norm
$\|f\| = \sup_{x \ge 0} |f(x)|  + \sup_{x \ge 0} \big| \int_0^x f(t) \, dt \big| + \|f\|_n$.
Moreover, an approach similar to that given in the proof of the preceding lemma shows that
each $Z_n$ is a Banach space under the latter norm. Finally, $Y$ is not closed in $X$.
Indeed, $Y$ contains the set $B$ of Lemma \ref{Riemann-lemma}, so $Y$ is dense in $X$.
But $Y \ne X$, because the function $\varphi :I \to \R$ defined as
$$
\varphi (x) = \left\{
\begin{array}{cl}
0 & \mbox{if } x \in \N_0\\
{1 \over \log (1+n)} & \mbox{if } x = 2n - {1 \over 2} \,\, (n \ge 1)\\
-{1 \over \log (1+n)} & \mbox{if } x = 2n + {1 \over 2} \,\, (n \ge 1)\\
\hbox{affine-linear} & \mbox{otherwise},
\end{array}
\right.
$$
is in $X$ but not in $Y$: each series $\sum_{n \ge 1} 1/\log^p (1+n)$ $(p > 0)$ diverges, $\varphi (x) \to 0$ and
$|\int_0^x \varphi | \le 1/\log (1+[x/2]) \to 0$ as $x \to +\infty$ ($[x]$ denotes the integer part of $x$).
Consequently, Theorem \ref{Kitson-Timoney} applies and $A = X \setminus Y$ is spaceable, as required.
\end{proof}

\subsection{The ``failure'' of the Lebesgue dominated convergence theorem.}

In this subsection we keep inside the setting of integrable functions, but focussing on {\it sequences}
of these functions. Results about interchanging of limits and integrals are well known, the most famous of them
being probably the Lebesgue dominated convergence theorem: if $(\Omega ,{\cal M},\mu )$ is a measure space,
$f_k:\Omega \to \R$ $(k \ge 1)$ are (Lebesgue) integrable functions, $f_k \to f$ a.e.~and $\sup_k |f_k|$ is
integrable, then ($f$ is integrable and) $\|f_k - f\|_1 \to 0$ (hence $\lim_{k \to \infty} \int_\Omega f_k \,d\mu =
\int_\Omega f \,d\mu$). Relaxing some of the hypotheses may drive to the failure of the conclusion. For instance,
for the Lebesgue measure on $\R$, we have that $f_k(x):= {k \over k^2+x^2} \to 0 =:f(x)$ for all $x \ne 0$, each $f_k$
is integrable with $\|f_k\|_1 = \pi$ for all $k$ (so $\sup_k \|f_k\|_1 = \pi < +\infty$) but $f_k \not\to f$ in $\| \cdot \|_1$.
By topologizing appropriately an adequate vector space, it will be shown that this
phenomenon is lineable in a strong sense, see Theorem \ref{dominated-convergence} below.
As in the example, our measure will be the Lebesgue measure on $\R$.

\vskip .15cm

For this, we consider the vector space $(\R^\R)^\N$ of sequences $(f_k)_{k \ge 1}$ of functions $\R \to \R$, as well
as the subspace of it given by
\begin{equation*}
\begin{split}
CBL_s := \{(f_k) \in (\R^\R)^\N: \,\, &\hbox{each } f_k \hbox{ is continuous, bounded and integrable,} \\
                                      &\|f_k\|_\infty \mathop{\longrightarrow}\limits_{k \to \infty} 0 \,
                                      \,\, \hbox{ and } \,\sup_k \|f_k\|_1 < +\infty\} .
\end{split}
\end{equation*}
It is a standard exercise to prove that $CBL_s$ becomes a Banach space when endowed with the norm
$\| (f_k) \| = \sup_k \|f_k\|_\infty + \sup_k \|f_k\|_1$. This space is, however, {\it not} \,separable, see Remark \ref{RemarkCBL}.3 below.
As usual, we have denoted $\|f\|_\infty = \sup_{x \in \R} |f(x)|$ for each $f:\R \to \R$.
In particular, $(f_k) \in CBL_s$ implies $f_k \to 0$ uniformly on $\R$. Next, consider the subset $\cal F$ of $CBL_s$
of sequences for which the dominated convergence theorem ``fails'', that is, the family
$$
{\cal F} := \{(f_k) \in CBL_s: \, \|f_k\|_1 \not\longrightarrow 0 \, \hbox{ as } \, k \to \infty\}.
$$
\begin{theorem}\label{dominated-convergence}
The set \,${\cal F}$ is spaceable  in \,$CBL_s$.
\end{theorem}

\begin{proof}
%With this aim, we apply Theorem \ref{BerOrd}. We use the obvious identification
%$(\R^\R)^\N = \R^{\R \times \N} = \R^{\N \times \R} = (\R^\N)^\R$, so that each sequence $(f_k) \in (\R^\R)^\N$
%can be thought as a function $\R \to \R^\N$. To use Theorem \ref{BerOrd} we set $\K = \R = \Omega$, $X = CBL_s$, $Z = \R^\N$
%(where $Z$ carries the product topology) and
%$S = \{(f_k) \in CBL_s: \, \|f_k\|_1 \mathop\longrightarrow\limits_{k \to \infty} 0\}$,
%so that ${\cal F} = X \setminus S$. In view of the definition of the norm $\| \cdot \|$, conditions
%(i) (with $(n_k) = \N$), (ii) (with $C=1$), (iii) (note that $S$ is a linear space) ... ...
We apply Wilansky's criterion given at the
beginning of Section 3: take $X = CBL_s$ and
$Y = \{(f_k) \in CBL_s: \, \|f_k\|_1 \mathop{\longrightarrow}\limits_{k \to \infty} 0\}$, so that ${\cal F} = X \setminus Y$.
Note that $Y$ is a vector subspace, and a standard argument yields that $Y$ is closed.
It is enough to exhibit a linearly independent sequence $\{\Phi_n = (f_{n,k})_{k \ge 1}: \, n \ge 1\} \subset X \setminus Y$.
With this aim, select infinitely many disjoint sequences $\{p(n,1) < p(n,2) < \cdots < p(n,k) < \cdots \}$ $(n=1,2, \dots )$
of natural numbers and define $f_{n,k}:\R \to \R$ as
$$
f_{n,k} (x) = \left\{
\begin{array}{ll}
\hskip -8pt (2/k)(x-p(n,k)) & \mbox{if } p(n,j) \le x < p(n,j) + {1\over 2} \,\,\, (1 \le j \le k)\\
\hskip -8pt (2/k)(p(n,k) + 1 - x) & \mbox{if } p(n,j) + {1 \over 2} \le x < p(n,j) + 1 \,\, (1 \le j \le k)\\
\hskip -8pt \, 0 & \mbox{otherwise}
\end{array}
\right.
$$
These functions satisfy $\|f_{n,k}\|_\infty = {1 \over k}$ and $\|f_{n,k}\|_1 = 1$, and their
supports are mutually disjoint. Hence the family $\{\Phi_n\}_{n \ge 1}$ is in $\cal F$ and is linearly independent, as required.
\end{proof}

\begin{remarks} \label{RemarkCBL}
{\rm 1. An alternative way of constructing the sequence $(\Phi_n)$ in the last proof is defining $f_{n,k}(x) := {k^n \over k^{2n} + x^2}$.

\noindent 2. Since $\cal F$ is spaceable, it is $\mathfrak{c}$-lineable. Moreover, we have ${\rm dim}\,(CBL_s) \le {\rm card} \,(CBL_s)
\le {\rm card} \, (C(\R )^\N ) \le {\rm card}\, ((\R^\N)^\N) = {\rm card}\, (\R ) = \mathfrak{c}$. Hence \break
${\rm dim}\,(CBL_s) = \mathfrak{c}$ and
$\cal F$ is {\it maximal lineable.} Another way to prove this is the following.
Let $Y$ as in the proof of Theorem \ref{dominated-convergence}.
Similarly to the proof of Lemma \ref{Riemann-lemma}, define for each $f:\R \to \R$ and each $a>0$ the function
$f^a:\R \to \R$ as $f^a = f$ on $[-a,a]$,
$f^a = 0$ for $|x|>a$, and $f^a$ affine-linear on $[-a,-a-1] \cup [a,a+1]$ with $f^a(a+1)=0=f^a(-a-1)$.
Then it is a standard exercise to check that the set $\{((f_1)^N, \dots ,(f_k)^N, 0,0,0, \dots ) \in (\R^\R)^\N:
\, f_1, \dots ,f_k$ are polynomials with {\it rational} coefficients and $k,N \in \N\}$
is a countable dense subset of $Y$. That is, $Y$ is a separable closed vector subspace
of the non-separable F-space $CBL_s$. By Remark \ref{Remark-nonseparable}, $\cal F$ is maximal lineable.

\noindent 3. We have not been able to demonstrate the (maximal) dense-lineability of $\cal F$; nevertheless, our conjecture is ``yes''.
Notice that not even the mere dense-lineability can be deduced from
Theorem \ref{simple dense-lineability}, because $CBL_s$ is not separable. Let us provide a simple proof of this fact.
Consider the mapping $T:(a_k)_k \in \ell_\infty \mapsto (f_k)_k \in CBL_s$, where
$$
f_k (x) = \left\{
\begin{array}{ll}
(2a_k/k)(x-j+1) & \mbox{if } j-1 \le x < j + {1\over 2} \,\,\,\, (1 \le j \le k)\\
(2a_k/k)(j-x) & \mbox{if } j + {1\over 2} \le x < j  \,\,\,\, (1 \le j \le k)\\
\,0 & \mbox{otherwise}
\end{array}
\right.
$$
Then $(1/2) \|(a_k)_k\|_{\ell_\infty} = (1/2)\sup_{k \ge 1} |a_k| \le \sup_{k \ge 1} |a_k/k| +
(1/2)\sup_{k \ge 1} |a_k| = \|T(a_k)_k\| \le 2 \sup_{k \ge 1} |a_k| = 2 \|(a_k)_k\|_{\ell_\infty}$. Hence $T$ is an isomorphism
between the nonseparable space $\ell_\infty$ and $T(\ell_\infty )$. Therefore $T(\ell_\infty )$ (and so $CBL_s$) is not separable.}
\end{remarks}

\subsection{Entire functions of fast growth and generalized Dirichlet spaces.}

We want to do here a new incursion into the complex plane. Let us consider the space ${\cal E} = H(\C )$ of entire functions, equipped with the
compact-open topology. Let $\varphi :[0,+\infty ) \to (0,+\infty )$ be an increasing function.
A simple application of the Weierstrass interpolation theorem (see e.g.~\cite[Chap.~15]{RudA}) yields the existence of
an entire function growing faster than $\varphi$, that is, such that $\varphi$ belongs to the set
$${\cal E}_\varphi := \left\{f \in {\cal E}: \, \limsup_{r \to +\infty} {\max \{|f(z)|: \, |z|=r\} \over \varphi (r)} = +\infty \right\} .$$
In fact, the dense-lineability of ${\cal E}_\varphi$ has already
been established, even with several additional properties (boundedness on large sets, vanishing on large sets as $z \to \infty$,
universality in the sense of Birkhoff, action of certain ope\-ra\-tors, etc), see for instance \cite{Arm,BerA0,BerA00,BCL2,Bon,Cal}.
As Theorem \ref{Growth-Dirichlet} below shows, ${\cal E}_\varphi$ enjoys stronger lineability properties.

\vskip .15cm

Next, we turn our attention to the disc $\D$ and consider the so-called {\it weighted Dirichlet spaces} given by
$$
{\cal S}_{\nu} = \left\{f(z) = \sum_{n=0}^\infty a_n z^n \in H(\D ): \,
\sum_{n=0}^\infty |a_n|^2 (n+1)^{2\nu} < +\infty \right\},
$$
where $\nu \in \R$. For instance, if $\nu = 0,-1/2,1/2$, then ${\cal S}_{\nu}$ is,
respectively, the classical Hardy space $H^2(\D )$, the Bergman
space $A^2(\D )$, and the Dirichlet space ${\cal D}$.
Each ${\cal S}_\nu$ becomes a Hilbert space under the inner product
$\langle \sum_{n=0}^{\infty}a_{n}z^{n},
\sum_{n=0}^{\infty}b_{n}z^{n} \rangle =
\sum_{n=0}^{\infty}a_{n}\overline {b_{n}}(n+1)^{2 \nu}$, see \cite{CoMc}. The corresponding norm is $\|f\|_\nu = (\sum_{n=0}^\infty |a_n|^2 (n+1)^{2\nu})^{1/2}$.
Observe that $S_\al \supsetneq S_\beta$ if $\beta > \al$. Then it is natural to ask what is the algebraic size of
$S_{\nu ,\, {\rm strict}} := S_\nu \setminus \bigcup_{a > \nu} S_a$.

\begin{theorem} \label{Growth-Dirichlet}
\begin{enumerate}
\item [\rm (a)] For every increasing function \,$\varphi :[0,+\infty ) \to (0,+\infty )$, the set \,${\cal S}_\varphi$ is maximal dense-lineable
                and spaceable in \,$\cal E$.
\item [\rm (b)] For every $\nu \in \R$, the set \,$S_{\nu ,\,{\rm strict}}$ is maximal dense-lineable and spaceable in \,$S_{\nu}$.
\end{enumerate}
\end{theorem}

\begin{proof}
(a) The Fr\'echet space $\cal E$ is metrizable and separable with ${\rm dim} \, ({\cal E}) = \mathfrak{c}$.
Denote $M(f,r) := \max \{|f(z)|: \, |z|=r\} = \max \{|f(z)|: \, |z| \le r\}$ for $f \in {\cal E}$, $r > 0$.
For each $\varphi$ as in the hypothesis, consider the auxiliary function $\psi (r) := e^{2r} + \varphi (r)$.
Since, obviously, ${\cal E}_\psi \subset {\cal E}_\varphi$, it is enough to prove the required lineability properties for ${\cal E}_\psi$.
Note that \,${\cal S}_\psi = {\cal E} \setminus Y$, where
$$
Y := \{f \in {\cal E}: \, \|f\| < \ve\} \quad \hbox{and} \quad \|f\| := \sup_{r > 0} M(f,r)/\psi (r).
$$
With the help of the inequality $M(f,N) \le \psi (N) \|f\|$
$(N=1,2,...)$ it is easy to see that $(Y,\| \cdot \|)$ is a Banach space such that the inclusion $j:Y \hookrightarrow {\cal E}$
is continuous. Given a polynomial $P$, there is a constant $C > 0$ with $M(P,r) \le Ce^r$ for all $r>0$.
It follows that
$$
\|P\| = \sup_{r > 0} {M(P,r) \over \psi (r)} \le \sup_{r > 0} {M(P,r) \over e^{2r}} \le \sup_{r > 0} {C \over e^r} = C < +\infty .
$$
Therefore $\{$polynomials$\} \subset Y$, so $Y$ is dense in $\cal E$.
Thus, $Y$ is not closed in $\cal E$ because $Y \ne {\cal E}$. Indeed, the Weierstrass interpolation theorem furnishes
a function $f \in {\cal E}$ with $f(n) = n \, \psi (n)$ for all $n \ge 1$, and plainly $f \notin Y$. Theorem \ref{Kitson-Timoney}
applies with $X = {\cal E}$, $Z_n = Y$ and $T_n = j$ for all $n \ge 1$, so yielding the spaceability of ${\cal E}_\psi$.
In particular, ${\cal E}_\psi$ is $\mathfrak{c}$-lineable. Moreover, $B := \{$polynomials$\}$ is a dense vector subspace of $\cal E$
with ${\cal E}_\psi + B \subset {\cal E}_\psi$. To see this note that, given $f \in {\cal E}_\psi$ and a polynomial $P$ as
before, one has
\begin{equation*}
\begin{split}
\sup_{r > 0} {M(f+P,r) \over \psi (r)} &\ge \sup_{r > 0} {M(f,r)-M(P,r) \over \psi (r)} \ge \sup_{r > 0} {M(f,r)- Ce^r \over \psi (r)} \\
&\ge \sup_{r > 0} {M(f,r)\over \psi (r)}  - \sup_{r > 0} {Ce^r \over e^{2r}} \ge \sup_{r > 0} {M(f,r) \over \psi (r)} - C = +\infty .
\end{split}
\end{equation*}
Consequently,
Theorem \ref{maximal dense-lineable} entails that ${\cal E}_\psi$ is maximal dense-lineable.

\vskip .15cm

\noindent (b) A similar scheme will be used here. It is evident that the polynomials form a dense vector subspace $B$ of the separable Banach space $S_\nu$ satisfying $S_{\nu ,\,{\rm strict}} + B \subset S_{\nu ,\,{\rm strict}}$. Since the spaceability of $S_{\nu ,\,{\rm strict}}$ implies its
maximal lineability, and then Theorem \ref{maximal dense-lineable} entails its maximal dense-lineability, it is enough to show spaceability,
for which Theorem \ref{Kitson-Timoney} is invoked again: first observe that
$S_{\nu ,\,{\rm strict}} = S_\nu \setminus \bigcup_{n \ge 1} S_{\nu +{1 \over n}} = S_\nu \setminus {\rm span} \, \big(\bigcup_{n \ge 1} S_{\nu +{1 \over n}} \big)$; then take $X = {\cal E}$,
$Z_n = (S_{\nu +{1 \over n}},\| \cdot \|_{\nu +{1 \over n}})$ and $T_n =$ the inclusion
$S_{\nu +{1 \over n}} \hookrightarrow S_\nu$ $(n \ge 1)$. Since, clearly,
each polynomial is in $Y := \bigcup_{a > \nu} S_a$, we will be done as soon as we exhibit a function $f \in S_\nu \setminus Y$ (because this would imply that $Y$ is not closed). To this end, we define
$$f(z) = \sum_{k=0}^\infty {z^k \over (k+1)^{\nu + {1 \over 2}} \cdot \log (k+1)}.$$
The proof is finished.
\end{proof}

\subsection{Peano curves.}

Lineability properties of families of functions $\varphi :\R \to \R$ that are surjective in very strong senses
(for instance, satisfying $\varphi (I) = \R$ for every interval $I$, and even with stronger conditions)
have recently studied by several authors (see \cite{AGS,ArS,BPS,Gam,GMSS,GRS}).
However, all of these functions are nowhere continuous. It is then natural to ask about {\it continuous} surjections.
As Albuquerque suggests in \cite{Alb}, one can adopt an even more general point of view and ask about
continuous surjections $\R^M \to \R^N$ $(M,N \in \N )$. Following \cite{Alb}, we denote
$$
{\cal S}_{M,N} = \{f:\R^M \longrightarrow \R^N : \,f \hbox{ \ is continuous and surjective}\}.
$$

\vskip .15cm

In 1890 G.~Peano surprised the mathematical world by constructing a {\it filling space curve}, that is,
a surjective continuous map $f:[0,1] \to [0,1]^2$. From this it is not difficult to construct a
surjective continuous function $\R \to \R^2$ as an extension of $f$. This extension together with an
inductive procedure is used in \cite{Alb} to show that ${\cal S}_{M,N} \ne \emptyset$ for every pair $(M,N)$.
Finally, by employing appropriate compositions, it is proved in \cite{Alb} that {\it each family \,${\cal S}_{M,N}$
\,is $\mathfrak c$-lineable.}

\vskip .15cm

We will improve here this result by adding topological properties. For this, we
consider the separable Fr\'echet space $C(\R^M, \R^N)$ of all continuous functions $\R^M \to \R^N$ under the compact-open topology.
By the Hahn--Mazurkiewicz theorem (see for instance \cite{HoY}), for every metrizable compact connected locally connected
topological space $X$ there is a continuous surjective mapping $[0,1] \to X$. In particular, if $I_N$ denotes the
$N$-cube $I_N=[0,1]^N$, there exists a continuous mapping $\varphi :[0,1] \to I_N$ with $\varphi ([0,1])=I_N$.
Therefore, the mapping
$$\Phi : (x_1, \dots ,x_M) \in S_0 \mapsto \varphi (x_1) \in I_N \eqno (3)$$
is continuous and satisfies $\Phi (S_0) = I_N$, where $S_0$ denotes the ``strip''
$S_0 = \{(x_1,\dots ,x_M) \in \R^M: \, 0 \le x_1 \le 1\} = [0,1] \times \R^{M-1}$,
meaning $S_0 = [0,1]$ if $M=1$.

\vskip .15cm

With the following theorem we conclude this paper. But before stating it, let us introduce a new family
that is smaller than ${\cal S}_{M,N}$. We denote
$$
{\cal S}_{M,N,\infty} = \{f \in C(\R^M, \R^N): \,f^{-1}(\{y\}) \hbox{ \ is } \hbox{unbounded for every } y \in \R^N\}.
$$
\begin{theorem} \label{Peano}
For each pair $(M,N)$ of natural numbers, the set \,${\cal S}_{M,N,\infty}$
\,{\rm (}hence the set \,${\cal S}_{M,N}${\rm )}
is maximal dense-lineable and spaceable in \,$C(\R^M, \R^N)$.
\end{theorem}

\begin{proof}
We make use of the well-known fact that the set \,$\cal P$ \,of functions $P=(P_1,\dots ,P_N):\R^M \to \R^N$ whose components
$P_1, \dots ,P_N$ are polynomials of $M$ variables is dense in \,$C(\R^M, \R^N)$. Fix $k \in \N$ and $P=(P_1,\dots ,P_N)$
as before. By Tietze's extension theorem (alternatively, a direct construction is not difficult) we obtain (and fix)
continuous functions $P_1[k],\dots ,P_N[k]:\R^M \to \R$ such that $P_j[k]=P_j$ on $B_k := \{(x_1, \dots ,x_M): \, x_1^2+\cdots+x_M^2 \le k\}$
and $P_j[k] = 0$ on $\R^M \setminus B_{k+1}^0$. Let denote \,$P[k] = (P_1[k], \dots ,P_N[k])$. Since each compact set $K \subset \R^M$ is
contained in some $B_k$ and the topology of \,$C(\R^M, \R^N)$ \,is that of uniform converge on compacta,
we have that the set \,${\cal P}_0 := \{P[k]: \, P \in {\cal P}, \, k \in \N\}$ \,is dense in \,$C(\R^M, \R^N)$.

\vskip .15cm

Suppose that we have already proved the spaceability of \,${\cal S}_{M,N,\infty}$. Then this set is $\mathfrak c$-lineable
because \,$C(\R^M, \R^N)$ \,is a separable infinite-dimensional F-space. Consider the set $B$ of continuous functions
$f:\R^M \to \R^N$ with bounded support $\sigma (f)$ (see (1)). On the one hand, $B$ is a dense vector subspace of \,$C(\R^M, \R^N)$.
On the other hand, ${\cal S}_{M,N,\infty} + B \subset {\cal S}_{M,N,\infty}$: indeed, if $f^{-1}(\{y\})$
is unbounded and $g \in B$ then $(f+g)^{-1}(\{y\}) \supset f^{-1}(\{y\}) \setminus \sigma (g)$, and the last set is still unbounded because
$\sigma (g)$ is bounded.
An application of Theorem \ref{maximal dense-lineable} yields the maximal dense-lineability of \,${\cal S}_{M,N,\infty}$.

\vskip .15cm

Consequently, our only task is to show the spaceability of \,${\cal S}_{M,N,\infty}$.
For this, we will use Theorem \ref{BerOrd} with $\Omega = \R^M$, ${\cal S}(A) = \ovl{A}$
(i.e.~${\cal S}(A)$ is the closure of $A$ in $\R^M$, so that ${\cal S}(\sigma (h)) = \ovl{\sigma (h)}$, the {\it topological support}
of a function $h:\R^M \to \R^N$), $X = C(\R^M, \R^N)$, $\K = \R$, $Z = \R^N$,
$S = C(\R^M, \R^N) \setminus {\cal S}_{M,N,\infty}=
\{f \in C(\R^M, \R^N): \,f^{-1}(\{y\})$ is bounded for some $y \in \R^N\}$
(we agree that $\emptyset$ is bounded), and
$$
\|f\| = \sum_{k=1}^\infty {1 \over 2^k} {\sup_{x \in B_k}\|f(x)\|_2  \over 1 + \sup_{x \in B_k} \|f(x)\|_2},
$$
where $\| \cdot \|_2$ denotes Euclidean norm in $\R^N$. Let us check conditions (i) to (v)
in Theorem \ref{BerOrd}:
\begin{enumerate}
\item[$\bullet$] (i) holds because uniform convergence on compacta implies pointwise convergence.
\item[$\bullet$] (ii) is true (with $C=1$) since the map $t \in [0,+\infty ) \mapsto {x \over 1+x} \in [0,+\infty )$
                  is increasing and $\|f(x) + g(x)\|_2  \ge \|f(x)\|_2$ for
                  all $x \in \R^N$ whenever $\sigma (f) \cap \sigma (g) = \emptyset$. Here $\sigma (h)$ denotes the support of $h$ as defined in (1).
\item[$\bullet$] (iii) is satisfied because $\al f = 0$ if $\al = 0$ and, provided that $\al \ne 0$, then $(\al f)^{-1}(\{y\}) = f^{-1}(\{\al^{-1} y\})$ for all $y \in \R^N$.
\item[$\bullet$] Assume that $f,g \in X$ and $\ovl{\sigma (f)} \cap \sigma (g) = \emptyset$. Then, in particular,
$\sigma (f) \cap \sigma (g) = \emptyset$, from which it follows that
$(f+g)^{-1}(\{y\}) = f^{-1}(\{y\}) \cup g^{-1}(\{y\})$ for all $y \in \R^N \setminus \{0\}$. Suppose that $f+g \in S$.
Then either there is $y \in \R^N \setminus \{0\}$ such that $(f+g)^{-1}(\{y\})$ is bounded, or $(f+g)^{-1}(\{y\})$ is unbounded
for all $y \ne 0$ but $(f+g)^{-1}(\{0\})$ is bounded.
In the first case, the last set identity forces $f^{-1}(\{y\})$ to be bounded, so $f \in S$.
Assume now that $(f+g)^{-1}(\{y\})$ is unbounded for all $y \ne 0$ but $(f+g)^{-1}(\{0\})$ is bounded.
We can suppose that $f^{-1}(\{y\})$ is unbounded for all $y \ne 0$ (otherwise, $f \in S$ and we would be done).
Therefore $\sigma (f)$ is unbounded. Let us prove that $f^{-1}(\{0\})$ is bounded (in which case $f \in S$).
By way of contradiction, assume that $f^{-1}(\{0\})$ is unbounded. Then $\partial f^{-1}(\{0\})$ is also unbounded
[indeed, if $\partial f^{-1}(\{0\})$ is bounded then there is $\al > 0$ such that $f(x) \ne 0$ for all $x$ with
$\|x\|_2 > \al$ due to the unboundedness of $\sigma (f)$ and the closedness of $f^{-1}(\{0\})$;
hence $f^{-1}(\{0\})$ would be bounded, which is absurd]. Now, we have:
$\partial f^{-1}(\{0\}) = \partial (\R^M \setminus f^{-1}(\{0\})) = \partial \sigma (f) \subset \ovl{\sigma (f)}
\subset \R^M \setminus \sigma (g) = g^{-1}(\{0\})$. We derive that if $x \in \partial f^{-1}(\{0\})$ then (since $f^{-1}(\{0\})$
is closed) $f(x) = 0 = g(x)$, so $(f+g)(x) = 0$. Therefore
$\partial f^{-1}(\{0\}) \subset  (f+g)^{-1}(\{0\})$, so $(f+g)^{-1}(\{0\})$ is also unbounded, which contradicts our assumption.
This yields (iv).
\item[$\bullet$] The idea underlying the proof of (v) is to construct continuous functions by shifting and scaling appropriately the function $\Phi$ given in (3).
Firstly, it is plain that there is sequence of points $(a_j) \subset \R^N$ satisfying $\R^N = \bigcup_{j \ge 1} (a_j+I_N)$.
For each $k \in \N_0$ and each $a \in \R^N$ we consider the mapping
$$\Phi_{k,a}:\left(\{k,k+1\} \cup \left[k+{1 \over 3},k+{2 \over 3}\right] \right) \times \R^{M-1} \to \R^N$$
given by $\Phi_{k,a} = 0$ on $\{k,k+1\} \times \R^{M-1} = \partial (k+S_0)$ and $\Phi_{k,a}(x_1, \dots ,x_M) = a + \varphi (3(x_1-k)-1)$
if $(x_1, \dots ,x_M) \in [k+{1 \over 3},k+{2 \over 3}] \times \R^{M-1}$. Tietze's extension theorem comes in our help to provide a continuous extension
$\Phi_{k,a}:k+S_0 \to \R^N$ (observe that Tietze's theorem can be applied to each component of $\Phi_{k,a}$). Note that $\Phi_{k,a} (k+S_0) \supset a + I_N$ for all $k \ge 0$. Since ${\rm card}\,(\N^3) = {\rm card}\,(\N )$, we can select $\N^2$-many pairwise disjoint
sequences $\{p(n,m,1) < p(n,m,2) < \cdots < p(n,m,j) < \cdots\}$ $(n,m \in \N )$ of natural numbers. For each $n \in \N$, define $f_n:\R^M \to \R^N$ by
$$
f_n(x) = \left\{
\begin{array}{cl}
\Phi_{p(n,m,j),a_j}(x) & \mbox{if } x \in p(n,m,j)+S_0 \,\,\, (m,j \in \N )\\
0 & \mbox{otherwise},
\end{array}
\right.
$$
Since $f_n = 0$ on each boundary $\partial (p(n,m,j)+S_0)$, we have that each $f_n$ is well defined and continuous.
Furthermore, for every $n \in \N$ and every $y \in \R^N \big( = \bigcup_{j \ge 1} (a_j+I_N) \big)$,
the set $f_n^{-1}(\{y\})$ possesses at least one point in every set
$\bigcup_{j \ge 1} (p(n,m,j)+S_0)$ $(m=1,2, \dots )$, so $f_n^{-1}(\{y\})$ is unbounded and $f_n \in {\cal S}_{M,N,\infty}$.
Finally, the supports of the functions $f_n$ $(n=1,2, \dots )$ satisfy $\ovl{\sigma (f_k)} \cap \sigma (f_n) = \emptyset$
for all $k \ne n$, because
$\sigma (f_n) \subset \bigcup_{m,j \ge 1}(p(n,m,j)+S_0^0)$ \,and the numbers $p(n,m,j)$ are pairwise different.
\end{enumerate}
This had to be shown.
\end{proof}

Theorem \ref{Peano} is best possible in terms of dimension because, as noticed in \cite{Alb},
there is no surjective continuous function $\R \to \R^\N$ (see \cite{Mun}), $\R^\N$ being the
space of real sequences endowed with the product topology.

\vskip .25cm

\noindent {\bf Acknowledgements.} The authors have been partially supported by the Plan
Andaluz de Investigaci\'on de la Junta de Andaluc\'{\i}a FQM-127
Grant P08-FQM-03543 and by MEC Grant MTM2012-34847-C02-01.

\bigskip

{\footnotesize  %\small

} %end of \footnotesize

\end{document}